\documentclass[a4paper,11pt]{article}

\usepackage[utf8]{inputenc} 
\usepackage{enumerate}

\usepackage[T1]{fontenc}

\usepackage{hyperref}       
\usepackage{url}            
\usepackage{booktabs}       
\usepackage{amsfonts,amsmath,amssymb,amsthm}       
\usepackage{nicefrac}       
\usepackage{microtype}      
\usepackage{lipsum}
\usepackage{xcolor}
\usepackage{graphicx,float}
\usepackage{tikz}
\usepackage[square,comma,compress,numbers]{natbib}
\usepackage[big]{layaureo}
\usepackage{perpage}
\MakePerPage{footnote}

\theoremstyle{plain}
\newtheorem{theorem}{Theorem}[section]
\newtheorem{lemma}[theorem]{Lemma}

\newtheorem{cor}[theorem]{Corollary}
\newtheorem{claim}[theorem]{Claim}
\newtheorem{thmbig}{Theorem}

\theoremstyle{definition}
\newtheorem{definition}[theorem]{Definition}
\newtheorem*{remark}{Remark}

\definecolor{myred}{RGB}{226,56,18}
\definecolor{myorange}{RGB}{228,139,0}
\definecolor{mygreen}{RGB}{4,215,17}
\definecolor{mygrey}{RGB}{180,180,180}

\def\Cx{\mathbb{C}}
\def\Chat{\widehat{\mathbb{C}}}
\def\D{\mathbb{D}}

\def\P{\mathcal{P}}

\def\N{\mathbb{N}}
\def\A{\mathrm{A}}

\begin{document}

\title{\textbf{\textsc{Holomorphic motions, natural families of entire maps, and multiplier-like objects for wandering domains}}}

\author{Gustavo R.~Ferreira\thanks{Email: \texttt{gustavo.r.f.95@gmail.com}}
 \;\& Sebastian van Strien \\
  \small{Department of Mathematics, Imperial College London}\\
  \small{London, UK}
}

\maketitle

\begin{abstract}
Structural stability of holomorphic functions has been the subject of much research in the last fifty years. Due to various technicalities, however, most of that work has focused on so-called finite-type functions (functions whose set of singular values has finite cardinality).  Recent developments in the field go beyond this setting. In this paper we extend Eremenko and Lyubich's result on natural families of entire maps to the case where the set of singular values is not the entire complex plane, showing under this assumption that the set $M_f$ of entire functions quasiconformally equivalent to $f$ admits the structure of a complex manifold (of possibly infinite dimension). Moreover, we will consider functions with wandering domains -- another hot topic of research in complex dynamics. Given an entire function $f$ with a simply connected wandering domain $U$, we construct an analogue of the multiplier of a periodic orbit, called a distortion sequence, and show that, under some hypotheses, the distortion sequence moves analytically as $f$ moves within appropriate parameter families.
\end{abstract}


\section{Introduction}\label{sec:intro}
We consider the iteration of holomorphic functions $f:\Cx\to\Cx$. As first shown by Fatou \cite{Fat19,Fat26} and Julia \cite{Jul18}, the complex plane is partitioned into an open set of ``regular'' dynamics, the \textit{Fatou set} (denoted $F(f)$), and a closed set of ``chaotic'' dynamics, the \textit{Julia set} (denoted $J(f)$) -- see, for instance, \cite{Bea91} or \cite{Ber93} for introductions to the subject. Both the Fatou and Julia sets are completely invariant under $f$, meaning that a connected component $U$ of the Fatou set -- called a \textit{Fatou component} -- is mapped by $f^n$ into another Fatou component, denoted $U_n$. If there exist $m > n\geq 0$ such that $U_m = U_n$, $U$ is said to be a \textit{(pre-)periodic} Fatou component. A Fatou component that is not pre-periodic is said to be a \textit{wandering domain}.

As with other areas of dynamical systems, much attention has been given to the question of structural stability in parameter families of holomorphic maps. A major tool in the area are \textit{holomorphic motions} (quasiconformal maps depending holomorphically on a parameter), which were introduced by Ma\~n\'e, Sad, and Sullivan in \cite{MSS83}, and have since become a standard part of the language of not only holomorphic dynamics, but also complex analysis in general; see e.g. \cite{Mit00,ALdM03,KSvS07,DL14,Zak16,LP17,LSvS20,BL22,ABF22}. In this paper, we give two novel applications of holomorphic motions to the study of parameter spaces of entire functions. First, in Subsection \ref{ssec:unf}, we use Mitra's \cite{Mit00} machinery of describing holomorphic motions of a closed set to show that equivalence classes of a wide variety of entire maps are complex manifolds of possibly infinite dimension, generalising results of Eremenko and Lyubich \cite{EL92}. Second, in Subsection \ref{ssec:ds}, we construct ``multiplier-like''  objects for entire functions with simply connected wandering domains, and study how these objects move within parameter families of entire maps -- but also how they can be used to parametrise such families.

\subsection{Natural families of entire maps}\label{ssec:unf}
When considering parameter families of holomorphic functions, one tends to consider \textit{natural} families, defined in Section \ref{sec:families} (see also \cite{ABF22}). This means that, in a sense, the largest possible parameter space containing an entire function $f$ is the set
\[ M_f := \{g = \psi\circ f\circ\varphi^{-1} : g\in E, \psi,\varphi\in QC(\Cx, \Cx)\}, \]
where $E$ denotes the topological vector space of entire functions (armed with the compact-open topology) and $QC(\Cx, \Cx)$ denotes the space of quasiconformal homeomorphisms of $\Cx$. The set $M_f$ is called the (quasiconformal) \textit{equivalence class} of $f$, and was shown to exhibit many properties that commend it as a ``natural'' parameter space; see e.g. \cite{Rem09,ER15,RGS17}. In particular, it was shown by Eremenko and Lyubich \cite[Section 2]{EL92} (see also \cite[Theorem 3.1]{GK86}) that if $f$ has finitely many singular values (critical values, asymptotic values, and accumulation points thereof) then $M_f$ is a complex manifold of dimension $q + 2$, where $q$ is the number of singular values. In particular, $M_f$ itself can be considered a natural family. Our first result extends this idea to much more general entire functions.
\\
\begin{thmbig}[Universal natural family]\label{thm:unf}
Let $f$ be an entire function, and assume that $f$ has at least two singular values. Let $F = S(f)\cup\{\infty\}$, where $S(f)$ is the set of singular values of $f$, and assume that $\Cx\setminus S(f)$ is non-empty. Then, there exists a Banach analytic manifold $T_f = T(F)\times(\Cx^*\times\Cx)^2$ (where $T(F)$ denotes the Teichm\"uller space of $F$; see Subsection \ref{ssec:teichmuller}) and a covering map $\Phi:T_f\to M_f\subset E$ satisfying the following properties:
\begin{enumerate}[(i)]
    \item $\Phi$ is continuous\footnote{Continuity here is meant considering $M_f\subset E$ with the topology of locally uniform convergence.} and the function $T_f\times\Cx\ni (\lambda, z)\mapsto \Phi_\lambda(z)\in\Cx$ is analytic.
    \item Let $(f_\lambda)_{\lambda\in M}$, where $M$ is a Banach analytic manifold, be a natural family containing  $f$. Then, there exists a holomorphic function $\phi:M\to T_f$ such that $f_\lambda = \Phi_{\phi(\lambda)}$. In other words, $\phi$ lifts the natural inclusion $M\ni \lambda\mapsto f_\lambda\in E$.
\end{enumerate}
In particular, $M_f$ admits a complex structure that turns $\Phi$ into a holomorphic map. Furthermore, $M_f$ is finite-dimensional if and only if $f$ has finitely many singular values.
\end{thmbig}
\par
\begin{remark}
If $f$ has exactly $q < +\infty$ singular values, then the manifold $T_f$ has dimension $q + 2$, as should be expected. See also \cite[Proposition A.1]{RGvS15} for a version of this result for real entire functions of finite type.
\end{remark}

The existence of $T_f$ helps us understand more general parameter spaces of entire functions -- though not necessarily in a straightforward manner. For instance, if $f$ has infinitely many singular values, then one can show that the manifold $T_f$ is modelled on a non-separable Banach space (see, for instance, \cite{Fle06}, or the proof of \cite[Theorem 6.6.1]{Hub06}), and is therefore non-separable. The Fr\'echet space $E$, however, is separable; consequently, $M_f$ is not locally closed in $E$ -- and thus not an embedded submanifold of $E$ -- except possibly when $f$ has finitely many singular values. On a more positive note, the proof of Theorem \ref{thm:unf} gives us the following corollary in the spirit of \cite[Corollary 5.4]{RGvS15}. For more powerful results on conjugacy classes (which are subsets of $M_f$), see also \cite[Theorem 4.11]{Lyu99} and \cite[Theorem B]{CvS23}.
\\
\begin{cor}\label{cor:connected}
For any $f\in E$, the equivalence class $M_f$ is connected, and in fact path-connected, in $E$.
\end{cor}
\par
Theorem \ref{thm:unf} also has the following consequence, which we believe to be of independent interest.
\\
\begin{cor}\label{cor:natfamily}
Let $f$ be an entire function such that $\#S(f) \geq 2$ and $\Cx\setminus S(f)$ is non-empty. Let $f_\lambda = \psi_\lambda\circ f\circ (\varphi_\lambda)^{-1}$ be entire functions parameterised by $\lambda\in M$, where $\psi_\lambda$ and $\varphi_\lambda$ are quasiconformal maps depending continuously on $\lambda$. Assume that, for $z\in S(f)$, the map $\lambda\mapsto \psi_\lambda(z)$ is holomorphic. Then, there exist quasiconformal maps $\psi_\lambda'$ and $\varphi_\lambda'$ depending holomorphically on $\lambda\in M$ such that $f_\lambda = \psi_\lambda'\circ f\circ(\varphi_\lambda')^{-1}$. In particular, $(f_\lambda)_{\lambda\in M}$ is a natural family.
\end{cor}

\subsection{Multiplier-like objects for wandering domains}\label{ssec:ds}
We now turn to the dynamics inside the Fatou set of entire functions. If $U$ is a periodic Fatou component of $f$, its internal dynamics are, for the most part, fully understood (see, for instance, \cite[Theorem 6]{Ber93}): a $p$-periodic Fatou component exhibits one of five types of internal dynamics, each with a clear topological model. In three of these five types (attracting, parabolic, and Siegel), the closure of the Fatou component contains a \textit{periodic point} $z_0$ whose \textit{multiplier} $(f^p)'(z_0)$ controls the dynamics of $U$ (the other two types, Baker domains and Herman rings, are not associated to periodic points). Wandering domains, on the other hand, cannot have periodic orbits in their closures, and so understanding their internal dynamics is a much more delicate endeavour -- one that has been carried out in \cite{BEFRS19,Fer22}. Based on this understanding of the different internal behaviours of wandering domains, we now consider the question: how do these different internal dynamics co-exist (if they do) inside parameter families?

If $f$ is part of a holomorphic family $(f_\lambda)_{\lambda\in M}$ (see Section \ref{sec:families} for a definition), where $M$ is complex manifold of possibly infinite dimension, then understanding how periodic points and their multipliers change with $\lambda$ is an important part of understanding how the dynamics change within $(f_\lambda)_{\lambda\in M}$. If the multiplier of the $p$-periodic point $z_0$ is not $1$, then by the implicit function theorem the equation $f_\lambda^p(z) = z$ has a solution $(\lambda, z(\lambda))$ in some neighbourhood $\Omega\subset M\times\Cx$ of $(\lambda_0, z_0)$, and the function $\lambda\mapsto z(\lambda)$ is holomorphic. In other words, we can locally ``track'' the periodic point $z_0$ as $\lambda$ changes. This means in particular that the function $\lambda\mapsto (f_\lambda^p)'(z(\lambda))$, the \textit{multiplier map}, is holomorphic in a neighbourhood of $\lambda_0$. It has been studied in great detail in the case of rational functions, and a common theme is that the multipliers of attracting periodic orbits can be used to parameterise the parameter families the functions belong to  -- see, for instance, \cite{DH85,Ree90,Mil12,Lev11}. This motivates the questions that drive this part of the paper -- which, if $f = f_{\lambda_0}\in (f_\lambda)_{\lambda\in M}$ is an entire function with a simply connected wandering domain $U$, can be phrased as:
\begin{enumerate}[(1)]
    \item Is there some ``multiplier-like'' object associated to $U$ and moving holomorphically with $\lambda$?
    \item Can we use this object to parameterise $(f_\lambda)_{\lambda\in M}$?
\end{enumerate}

If we are assuming that $J(f_\lambda)$ moves holomorphically (see Definition \ref{def:Jstable}) in the family $(f_\lambda)_{\lambda\in M}$ and $H(\lambda, z)$ is the holomorphic motion of its Julia set, an immediate consequence is that $f_\lambda$ also has a simply connected wandering domain $U_\lambda = H(\lambda, U)$ for $\lambda$ close enough to $\lambda_0$. This is where we introduce our analogue of the multiplier map: the \textit{distortion map} (distortion sequences, the direct analogue of the multiplier, are introduced in Definition \ref{def:DS}). We show that, under appropriate assumptions, it is holomorphic in the appropriate domain (see Section \ref{sec:families} for relevant definitions, and Subsection \ref{ssec:conj} for an exposition on conjugations on Banach analytic manifolds):
\\
\begin{thmbig}[A holomorphic deformation defines a distortion sequence]\label{thm:DS}
Let $(f_\lambda)_{\lambda\in M}$, where $M$ is a Banach analytic manifold, be a holomorphic family of entire functions. Assume that $f = f_{\lambda_0}$ has a simply connected wandering domain $U$, and let $p\in U$. Let $\alpha$ be a distortion sequence for $f$ at $p$. Assume that $J(f_\lambda)$ moves holomorphically over some neighbourhood $\Lambda\subset M$ of $\lambda_0$. Then, there exist a neighbourhood $\Lambda'\subset\Lambda$ of $\lambda_0$, a holomorphic function $p\colon\Lambda'\to\Cx$ and a holomorphic map $\A\colon\Lambda'\times\overline{\Lambda'}\to\ell^\infty$, called a \emph{distortion map} (of $\alpha$ over $\Lambda'$), such that:
\begin{itemize}
    \item For every $\lambda\in\Lambda'$, $p(\lambda)$ belongs to the wandering domain $U_\lambda$ of $f_\lambda$.
    \item For every $\lambda\in \Lambda'$, $\A(\lambda, \overline{\lambda})$ is a distortion sequence for $f_\lambda$ at $p(\lambda)$.
    \item $p(\lambda_0, \overline{\lambda_0}) = p$ and $\A(\lambda_0, \overline{\lambda_0}) = \alpha$.
\end{itemize}
\end{thmbig}
\par
Let us make some comments about Theorem \ref{thm:DS}. First, the notation $\overline{\Lambda'}$ denotes the complex conjugate of $\Lambda'$, which is in practice the same Banach manifold with a different complex structure; an element $\overline\lambda\in{\Lambda'}$ is, in fact, $\lambda$ viewed as an element of $\overline{\Lambda'}$. The ``identity'' map $\Lambda'\ni\lambda\mapsto\overline{\lambda}\in\overline{\Lambda'}$ is anti-holomorphic\footnote{This might seem counterintuitive when one thinks of conjugation as it is usually understood in $\Cx$. The confusion comes from a closely related but distinct notion of conjugation that does not exist for every complex manifold, namely an anti-holomorphic involution. However, not every complex manifold admits such an involution; see, for instance, \cite{CF19}.}; see Subsection \ref{ssec:conj} for more details. Second, Theorem \ref{thm:DS} also has the following immediate consequence:
\\
\begin{cor}\label{cor:real}
Let $(f_\lambda)_{\lambda\in M}$ be a holomorphic family of entire functions parameterised by the Banach analytic manifold $M$. Assume that $f_{\lambda_0}$ has a simply connected wandering domain $U$, and let $p\in U$. Then, if $J(f_\lambda)$ moves holomorphically for $\lambda$ in some neighbourhood $\Lambda\subset M$ of $\lambda_0$, there exists a holomorphic function $\Lambda'\ni\lambda\mapsto p_\lambda\in U_\lambda$ (where $\Lambda'\subset\Lambda$ is a smaller neighbourhood of $\lambda_0$) such that $p_{\lambda_0} = p$ and the distortion sequence of $f_\lambda$ at $p_\lambda$ moves real-analytically with $\lambda\in\Lambda'$.
\end{cor}
\par
Given that it is not hard to show that each ``coordinate'' of the distortion sequence $\alpha$ moves real-analytically with $\lambda$ (see Section \ref{sec:DS}), Theorem \ref{thm:DS} seems like a rather convoluted way of obtaining Corollary \ref{cor:real}. However, a sequence of real-analytic functions, even uniformly bounded ones, does not necessarily define a real-analytic function into $\ell^\infty$; see Lemma \ref{lem:linfty} and the following discussion.

Finally, combining Theorem \ref{thm:DS} with Benini \textit{et al.}'s classification of simply connected wandering domains (see Theorem \ref{thmout:befrs} in Section \ref{sec:prelim} for a statement) has the following consequence -- which, for periodic domains, comes from the multiplier map and the classification of periodic Fatou components. This highlights how the distortion map is a ``wandering'' counterpart to the multiplier map.\\
\begin{cor}\label{cor:contr}
Let $f$, $U$, $M$, and $\alpha$ be as in Theorem \ref{thm:DS}. If $\|\alpha\|_\infty < 1$, then there exists a neighbourhood $\Lambda''\subset M$ of $\lambda_0$ such that, for $\lambda\in\Lambda$, the wandering domain $U_\lambda$ of $f_\lambda$ is contracting.
\end{cor}
\par
This is an answer to question (1); let us now turn to question (2). We show that, for certain kinds of functions with certain kinds of wandering domains, it is possible to parametrise a deformation of the wandering domain by perturbing its distortion sequence. In order to understand what kind of functions we are referring to, we turn to a method originally described by Herman.

Let $g:\Cx^*\to\Cx^*$ be a holomorphic function with a simply connected, forward invariant Fatou component $V$. Herman \cite{Her84} described how to find an entire function $f:\Cx\to\Cx$ such that $\exp\circ f = g\circ\exp$, i.e. $f$ lifts $g$, and any component of $\exp^{-1}(V)$ is a wandering domain of $f$. This motivates the following definition:\\
\begin{definition}\label{def:Herman}
Let $f$ be an entire function with a simply connected wandering domain $U$. We say that $U$ is an \textit{attracting} (resp. \textit{parabolic}, \textit{Siegel}) \textit{Herman-type} wandering domain if there exists a holomorphic function $g:\Cx^*\to\Cx^*$ such that $g\circ\exp = \exp\circ f$ and $V = \exp(U)$ is an attracting domain (resp. parabolic domain, Siegel disc) of $g$.
\end{definition}
\begin{remark}
It follows from Definition \ref{def:Herman} that, if $f$ has a Herman-type wandering domain, then $f$ satisfies $f(z + 2\pi) = f(z) + 2\pi i\cdot n$ for some $n\in\mathbb{Z}^*$. This, in turn, happens if and only if $f(z) = nz + h(e^z)$ where $h$ is an entire function (see, for instance, \cite[Proposition 7]{Kis99}). Thus, functions with Herman-type wandering domains are relatively simple to construct or identify.
\end{remark}

By considering the Teichm\"uller space of functions with Herman-type wandering domains, Fagella and Henriksen \cite{FH09} showed that such functions are structurally stable in a natural family parametrised by an abstract infinite-dimensional manifold (its Teichm\"uller space). Here, we give an explicit construction of such a manifold, and show that the associated distortion map is non-constant.
\\
\begin{thmbig}[A distortion sequence defines a holomorphic deformation]\label{thm:Herman}
Let $f$ be an entire function with an attracting Herman-type wandering domain, and let $M = \{\lambda\in\ell^\infty : \|\lambda\|_\infty < 1\}$. Then, there exists a natural family $(f_\lambda)_{\lambda\in M}$ such that $J(f_\lambda)$ moves holomorphically over $M$ and the distortion map $\A:M\times\overline{M}\to\ell^\infty$ is non-constant.
\end{thmbig}
\par
The outline of the paper is as follows. In Section \ref{sec:prelim}, we go through necessary concepts and results on the internal dynamics of wandering domains, holomorphic functions in Banach spaces, Banach analytic manifolds and conjugations, and quasiconformal maps and Teichm\"uller theory. Then, in Section \ref{sec:families}, we introduce existing results and concepts about parameter families of holomorphic functions and prove Theorem \ref{thm:unf}. Section \ref{sec:DS} introduces the notion of distortion sequences and proves Theorem \ref{thm:DS}, and finally in Section \ref{sec:Herman} we prove Theorem \ref{thm:Herman}. We notice that the results in Sections \ref{sec:DS} and \ref{sec:Herman} do not depend on the proof of Theorem \ref{thm:unf}; we only prove it first in order to keep similar sections together.

\par\noindent\textsc{Acknowledgements.} The first author thanks N\'uria Fagella and Lasse Rempe for many illuminating discussions. We thank the referee for pointing out mistakes in a previous version of Theorem \ref{thm:unf} and suggesting how to show that $\Phi$ is a covering, and for many other helpful comments that helped us vastly improve this paper. The first author acknowledges financial support from the London Mathematical Society in the form of the Early Career Fellowship ECF-2022-16, and from the Spanish State Research Agency through the Mar\'ia de Maeztu Program for Centers and Units of Excellence in R\&D (CEX2020-001084-M).

\section{Preliminaries}\label{sec:prelim}
\subsection{Internal dynamics of simply connected wandering domains}
Let $f$ be an entire function with a simply connected wandering domain $U$. For any pair of distinct points $z$ and $w$ in $U$, the sequence $\left(d_{U_n}(f^n(z), f^n(w))\right)$ of hyperbolic distances (see \cite{BM06} for an introduction to the hyperbolic metric of plane domains and its properties) is, by the Schwarz--Pick Lemma \cite[Theorem 6.4]{BM06}, decreasing, and hence has a limit $c(z, w)$ which is either zero or positive. It was shown by Benini \textit{et al.} in \cite{BEFRS19} that this property is (mostly) independent of $z$ and $w$, and so is the question whether this limit is reached. More specifically:
\begin{theorem}[Benini \textit{et al.} \cite{BEFRS19}]\label{thmout:befrs}
Let $f$ be an entire function with a simply connected wandering domain $U$. Let $Z_0\in U$, and let $E = \{(z, w)\in U\times U : \text{$f^k(z) = f^k(w)$ for some $k\in\N$}\}$. Then, exactly one of the following holds.
\begin{enumerate}[(1)]
    \item $d_{U_n}\left(f^n(z), f^n(w)\right)\to 0$ for all $z, w\in U$, and we say that $U$ is \emph{contracting}.
    \item $d_{U_n}\left(f^n(z), f^n(w)\right)\to c(z, w) >  0$ and $d_{U_n}\left(f^n(z), f^n(w)\right) \neq c(z, w)$ for all $(z, w)\in U\times U\setminus E$, and we say that $U$ is \emph{semi-contracting}.
    \item There exists $N\in\N$ such that, for all $n\geq N$, $d_{U_n}\left(f^n(z), f^n(w)\right) = c(z, w) > 0$ for all $(z, w)\in U\times U\setminus E$, and we say that $U$ is \emph{eventually isometric}.
\end{enumerate}
Furthermore, for any $z\in U$ and $n\in\N$, let
\[ \alpha_n(z) := \|Df(f^{n-1}(z))\|_{U_{n-1}}^{U_n} \]
denote the hyperbolic distortion of $f$ at $f^{n-1}(z)$. Then:
\begin{itemize}
    \item $U$ is contracting if and only if $\sum_{n\geq 1}(1 - \alpha_n(z)) = +\infty$.
    \item $U$ is eventually isometric if and only if $\alpha_n(z) = 1$ for all sufficiently large $n$.
\end{itemize}
\end{theorem}

The sequence $(\alpha_n(z))_{n\in\N}$ is a distortion sequence for $f$ at $z$ (see Section \ref{sec:DS}) with the property of being positive -- in particular, the dynamics of $U$ are intimately related to its distortion sequences.

\subsection{Holomorphic functions on Banach spaces}
The theory of holomorphic functions on Banach spaces and manifolds is rich, and shows both similarities and differences to complex analysis on $\Cx$ (or even $\Cx^n$); we refer the reader to \cite{Muj86} for an overview of the subject, and particularly for proofs of the results given here. The first and most important result we will state is Hartogs' separate analyticity\footnote{Unless otherwise specified, we always use analyticity to mean complex analyticity. As such, we use the words analytic and holomorphic interchangeably.} theorem \cite[Theorem 36.1]{Muj86}:
\begin{theorem}[Hartogs Theorem]\label{lem:hartogs}
Let $X$, $Y$, and $Z$ be complex Banach spaces, and let $U\subset X$ and $V\subset Y$ be open sets. Then, a function $f:U\times V\to Z$ is holomorphic if and only if, for every $x_0\in U$ and $y_0\in V$, the functions $y\mapsto f(x_0, y)$ and $x\mapsto f(x, y_0)$ are holomorphic.
\end{theorem}

Next, we have a very useful condition for analyticity of functions into $\ell^\infty$; see \cite[Exercise 8.H]{Muj86}.
\begin{lemma}\label{lem:linfty}
Let $X$ be a complex Banach space, let $U\subset X$ be open, and let $f_n:U\to \Cx$ be holomorphic functions. Suppose that, for each compact subset $K\subset U$, there exists $M_K > 0$ such that $\sup_{z\in K} |f_n(z)| < M_K$ for all $n\in\N$. Then, the function $F:U\to\ell^\infty$ given by
\[ F(z) = \left(f_n(z)\right)_{n\in\N} \]
is holomorphic.
\end{lemma}

It is important to notice that Lemma \ref{lem:linfty} does not hold if we only assume real-analyticity. Indeed, let $f_n\colon[0, 1]\to \Cx$ be given by $f_n(x) = \exp(2n\pi i x)$. Then, clearly, the sequence $(f_n)_{n\in\mathbb{N}}$ is uniformly bounded; however, the function $F\colon[0, 1]\to\ell^\infty$ defined by $F(x) = (f_n(x))_{n\in\mathbb{N}}$ is not even continuous, let alone real-analytic.

Since analyticity is a local property, it follows immediately that all the results discussed in this section are also valid for analytic functions between Banach analytic manifolds.

\subsection{Banach analytic manifolds and conjugations}\label{ssec:conj}
A Banach analytic manifold is, in a sense, the ``obvious'' way to define an infinite-dimensional complex manifold. More precisely, it is a Hausdorff topological space $M$ endowed with an open cover $\{U_\alpha\subset M\}_{\alpha\in A}$ and homeomorphisms $\varphi_\alpha:U_\alpha\to\varphi_\alpha(U_\alpha)\subset V$, where $V$ is a complex Banach space, and such that the transition maps $\varphi_\beta\circ\varphi_\alpha^{-1}:\varphi_\alpha(U_\alpha\cap U_\beta)\to \varphi_\beta(U_\alpha\cap U_\beta)$ are biholomorphic for any $\alpha, \beta\in A$. The functions $\varphi_\alpha$ are called the \textit{charts} of $M$.

For some $v = (z_1, \ldots, z_n)\in \Cx^n$, the meaning of $\overline v$ is relatively immediate: it is $(\overline{z_1}, \ldots, \overline{z_n})$. In a more abstract Banach analytic manifold, however, the meaning of conjugation becomes less clear. To understand it, we must first consider what it means to conjugate elements of complex Banach spaces.

If $V$ is a complex Banach space, we define a \textit{conjugation} to be an anti-linear map $\sigma:V\to V$ such that $\sigma\circ\sigma(v) = v$ for all $v\in V$. It is not obvious that every complex Banach space $V$ has a conjugation; however, one can always be constructed by dividing $V$ into ``real'' and ``imaginary'' parts\footnote{This is called a \textit{real structure} on $V$.}. More specifically, take a Hamel basis for $V$, and consider its $\mathbb{R}$-span $V_\mathbb{R}$. It is easy to verify that $V = V_\mathbb{R}\oplus iV_\mathbb{R}$; thus, any $v\in V$ can be written as $v = x + iy$, where $x$ and $y$ are in $V_\mathbb{R}$, and a well-defined conjugation is given by $\sigma(x + iy) = x - iy$. Taking different Hamel bases gives us different conjugations on $V$ (for instance, we can start with any $z\in\Cx$ with $|z| = 1$ and obtain a conjugation on $\Cx$ by reflecting across the $\mathbb{R}$-span of $z$).

On a Banach analytic manifold $M$, we conjugate by changing the complex structure of $M$. In other words, if $\{\psi_\alpha:U_\alpha\subset M\to V\}_{\alpha\in A}$ are charts for $M$, then we define $\overline{M}$ as the manifold with charts $\{\sigma\circ\psi_\alpha:U_\alpha\to V\}_{\alpha\in A}$, where $\sigma:V\to V$ is a conjugation. Notice that the identity is an anti-holomorphic map from $M$ to $\overline{M}$.

\subsection{Quasiconformal maps and Teichm\"uller theory}\label{ssec:teichmuller}
As usual, one of our main tools will be the measurable Riemann mapping theorem of Ahlfors and Bers \cite{AB60}; see \cite[Theorem 4.6.1 and Proposition 4.7.5]{Hub06} and \cite[Theorem 4.4.1]{FM} for the version given here. Here and throughout, for a quasiregular map $\varphi$, we denote by $\varphi^*\mu$ the pullback of the almost complex structure $\mu$ by $\varphi$, and by $\mu_0$ the standard almost complex structure. Also here and throughout, we denote by $M(X)$ the unit ball in the Banach space $L^\infty(X)$ of measurable, a.e. bounded functions on $X\subset\Cx$.
\begin{theorem}\label{lem:MRMT}
    Let $\mu\in M(\Cx)$. Then, the following hold:
\begin{enumerate}[(1)]
    \item There exists a quasiconformal map $\varphi:\Cx\to\Cx$ such that $\varphi$ has Beltrami coefficient $\mu$, i.e., $\varphi^*\mu_0 = \mu$. Furthermore, $\varphi$ is unique up to postcomposition with an automorphism of $\Cx$.
    \item Let $\varphi^\mu$ denote the quasiconformal homeomorphism of $\Cx$ such that $\varphi^*\mu_0 = \mu$ normalised by $\varphi^\mu(0) = 0$, $\varphi^\mu(1) = 1$, and $\varphi^\mu(\infty) = \infty$. Then, the map $M(\Cx)\mapsto QC(\Cx, \Cx)$ given by $\mu\mapsto \varphi^\mu$ is analytic, in the sense that it is continuous relative to the compact-open topology and for each $z\in\Cx$ the map $\mu\mapsto\varphi^\mu(z)$ is analytic.
\end{enumerate}
\noindent Item (1) holds with $\Cx$ replace by $\D$. Item (2) holds with $\Cx$ replaced by $\D$ and analytic replaced by real-analytic.
\end{theorem}

Notice that Theorem \ref{lem:MRMT}(2) is not the usual statement of the integrating map's holomorphic dependence on parameters. Usually, one considers the Beltrami coefficient $\mu$ to depend holomorphically on some parameter $\lambda$ moving in some complex manifold; here, however, the Beltrami coefficient $\mu$ itself is considered as a variable moving in the Banach analytic manifold $M(\Cx)$.

The analytic dependence of mappings on Beltrami coefficients has something of a converse \cite[Lemma 4.8.15]{Hub06}:
\begin{lemma}\label{lem:MRMTconv}
Let $M$ be a Banach analytic manifold, $U$ an open subset of $\Cx$, and $F:M\times U\to\Cx$ a continuous map. Let $h_t(z) := H(t, z)$. Suppose that $h_t$ is quasiconformal for all $t\in M$, and that for every $z\in U$ the map $t\mapsto h_t(z)$ is analytic. Then, the Beltrami coefficient
\[ \mu_t := (h_t)^*\mu_0 \]
is analytic as a map $M\mapsto L^\infty(\Cx)$.
\end{lemma}

Families of functions $h_t$ as in Lemma \ref{lem:MRMTconv} bring us, finally, to the definition of a holomorphic motion:
\begin{definition}
Let $X\subset \Cx$. A \textit{holomorphic motion} of $X$ over a Banach analytic manifold $M$ with basepoint $\lambda_0\in M$ is a mapping $H:M\times X\to\Cx$ such that:
\begin{enumerate}
    \item For each $x\in X$, the map $\lambda\mapsto H(\lambda, x)$ is analytic;
    \item For each $\lambda\in M$, the map $x\mapsto H(\lambda, x)$ is injective; and
    \item $H(\lambda_0, x) = x$.
\end{enumerate}
\end{definition}

This definition was introduced by Ma\~n\'e, Sad, and Sullivan in \cite{MSS83}, who called them ``analytic families of injections'', and has since become a cornerstone of holomorphic dynamics. A major property of holomorphic motions is the following (see \cite[Theorem 5.2.3]{Hub06}):
\begin{lemma}\label{lem:lambda1}
Let $M$ be a Banach analytic manifold, and let $X\subset\Cx$. If $H:M\times\Cx\to\Cx$ is a holomorphic motion, then for every $\lambda\in M$ the map $X\ni x\mapsto H(\lambda, x)\in \Cx$ is quasiconformal.
\end{lemma}

A major reason for the popularity and usefulness of holomorphic motions is that, in many cases, much more can be said: a holomorphic motion of $X$ can be extended to a holomorphic motion of $\overline{X}$ (see \cite[Theorem 5.2.3]{Hub06}), and sometimes even to a holomorphic motion of $\Cx$. To understand when and how this is possible, we must take a slight detour.

Given a closed set $F\subset\Cx$, an important question is to consider the family of all possible holomorphic motions of $F$. It turns out that the answer is deep, and requires a foray into Teichm\"uller theory. Let us start with some definitions (properly tailored to our context):
\begin{definition}
Let $S\subset\Cx$ be open, and let $\psi:S\to\psi(S)\subset\Cx$ and $\varphi:S\to\varphi(S)\subset\Cx$ be quasiconformal maps. We say that $\psi$ and $\varphi$ are equivalent (or \textit{Teichm\"uller equivalent}) if there exists a conformal map $h:\psi(S)\to\varphi(S)$ such that the quasiconformal map $\eta = \varphi^{-1}\circ h\circ\psi:S\to S$ is isotopic to the identity relative to $\partial S$. In other words, $\eta$ extends continuously to $\partial S$ and there exists a continuous map $H\colon[0, 1]\times \overline{S}\to\overline{S}$ such that $H(0, z) = z$ for all $z\in S$, $H(1, z) = \eta(z)$ for all $z\in S$, and $H(t, z) = z$ for all $t\in[0, 1]$ and $z\in\partial S$.
\end{definition}
\begin{definition}
Let $S$ be a plane domain. Then, its \textit{Teichm\"uller space} $\mathcal{T}(S)$ is the set of all possible quasiconformal maps $\varphi:S\to\varphi(S)\subset\Chat$ modulo Teichm\"uller equivalence.
\end{definition}

The fact that $\mathcal{T}(S)$ is a complex manifold for any plane domain $S$ (or, more generally, for any Riemann surface; see e.g. \cite[Theorem 6.5.1]{Hub06}) is a powerful result with many applications, and so is the fact that this idea can be generalised to \textit{unions} of mutually disjoint plane domains (see \cite[Section 5.3]{MS98} or \cite[Section 5]{Mit00} for details on how to do this).

Now, let $F\subset\Chat$ be closed, and assume that $\{0, 1, \infty\}\subset F$. If $\psi, \varphi:\Cx\to\Cx$ are quasiconformal maps fixing $0$, $1$, and $\infty$, we say (in the spirit of the previous definitions) that $\psi$ and $\varphi$ are \emph{$F$-equivalent} if $\psi\circ\varphi^{-1}$  is isotopic to the identity relative to $F$ (in particular, $\varphi|_F = \psi|_F$). We define the \emph{Teichm\"uller space of $F$, denoted $T(F)$}, to be the set of $F$-equivalence classes of quasiconformal homeomorphisms of $\Chat$ fixing $0$, $1$, and $\infty$. We will see that the space $T(F)$ holds the answer to our problem of describing all possible holomorphic motions of $F$; let us take a closer look at it.

It follows from the Ahlfors--Bers theorem that the map $P:M(\Cx)\to T(F)$ assigning to $\mu$ the $F$-equivalence class of its unique normalised integrating map $\varphi^\mu$ is a well-defined function, but this by itself does not tell us much. Far more useful is the following characterisation of $T(F)$:
\begin{lemma}[\cite{Mit00}, Corollaries 5.3, 6.1, and 6.2]\label{lem:TE}
For any closed set $F\subset\Chat$ containing $0$, $1$, and $\infty$, the following hold.
\begin{enumerate}
    \item $T(F)\simeq \mathcal{T}(\Chat\setminus F)\times M(F)$, where (as above) $M(F)$ denotes the unit ball in $L^\infty(F)$.
    \item There exists $\pi:M(\Chat\setminus F)\to\mathcal{T}(\Chat\setminus F)$ such that the map $P_F:M(\Cx)\to\mathcal{T}(\Chat\setminus F)\times M(F)$ given by $P_F(\mu) = (\pi(\mu), \mu|_F)$ is a split submersion.
    \item The function $P_F$ satisfies $P_F(\mu) = P_F(\nu)$ if and only if $P(\mu) = P(\nu)$.
\end{enumerate}
In particular, $T(F)$ admits a complex structure that makes $P_F$ into a holomorphic split submersion.
\end{lemma}

This manifold enabled Mitra to fully describe the holomorphic motions of a closed set $F\subset\Cx$ \cite{Mit07} via the following universality property.

\begin{theorem}\label{lem:mitra}
For any closed set $F\subset\Chat$ containing $0$, $1$, and $\infty$, there exists a universal holomorphic motion $\Psi_F:T(F)\times F\to\Chat$ with the following property: if $H:M\times F\to\Chat$ is any other holomorphic motion of $F$ over a simply connected Banach analytic manifold $M$, then there exists a unique holomorphic map $\phi:M\to T(F)$ such that $\Psi_F(\phi(\lambda), z) = H(\lambda, z)$ for all $(\lambda, z)\in M\times F$. If $M$ is not simply connected, such a lift exists if and only if each $H(\lambda, \cdot)$ extends to a homeomorphism of $\Chat$ that depends continuously on $\lambda$.
\end{theorem}

A ``pointwise'' extension of a holomorphic motion as above is called a \textit{trivialization} of the holomorphic motion. If, in addition, the extending homeomorphism of $\Chat$ is quasiconformal for every $\lambda\in M$, it is called a \textit{quasiconformal} trivialization.

The existence of this universal holomorphic motion also enabled Mitra to formulate Bers and Royden's \cite{BR86} result on extension of holomorphic motions over a unit ball in a very natural way \cite[Theorem B]{Mit00}:
\begin{theorem}\label{lem:lambda2}
Let $F\subset \Chat$ be a closed set, and let $H\colon M\times F\to \Chat$ be a holomorphic motion with basepoint $\lambda_0\in M$, where $M$ is a simply connected Banach analytic manifold. Then there is a holomorphic motion $\tilde H\colon U\times\Chat\to\Chat$, where $U = \{\lambda\in M\colon \rho_M(\lambda, \lambda_0) < r\}$, $\rho_M$ is the Kobayashi pseudo-metric on $M$, and $r > 0$ is a univeral constant, that extends $H$. Furthermore, there is a \emph{unique} such extension with a harmonic Beltrami coefficient\footnote{This is a technical condition that will not show up again in this paper; see \cite[Section 2.2]{Mit00}.} in each component of $\Chat\setminus F$.
\end{theorem}
Since every Banach analytic manifold is locally simply connected, Theorem \ref{lem:lambda2} in fact enables us to extend any holomorphic motion when restricted to a suitably small neighbourhood of its basepoint at the cost of a universal lower bound on the size of said neighbourhood. For stronger results in multiply connected Banach manifolds, see \cite{JM18}.

\section{Parameter families of entire functions}\label{sec:families}
\subsection{Holomorphic and natural families}
In this subsection, we give an overview of the main concepts used to discuss parameter families of holomorphic functions. Since the functions themselves are holomorphic, we can ask for parameter dependence to be holomorphic as well, bringing us to the concept of a holomorphic family. By Hartogs' Theorem, this means that we can define a holomorphic family of entire functions $(f_\lambda)_{\lambda\in M}$, where $M$ is a Banach analytic manifold, as a holomorphic function $F:M\times\Cx\to\Cx$ with $F(\lambda, z) = f_\lambda(z)$. Since we are concerned with the quasiconformal equivalence class $M_f$ of a given entire function $f$, it also makes sense to consider a different type of parameter family:

\begin{definition}\label{def:natfam}
Let $M$ be a Banach analytic manifold. A \textit{natural family} of entire functions over $M$ is a family $(f_\lambda)_{\lambda\in M}$ such that $f_\lambda = \psi_\lambda\circ f\circ\varphi_\lambda^{-1}$, where $f = f_{\lambda_0}$ is an entire function and $\psi_\lambda$ and $\varphi_\lambda$ are quasiconformal homeomorphisms of $\Cx$ depending holomorphically on $\lambda\in M$.
\end{definition}

It is easy to show (see the proof of Theorem \ref{thm:unf}) that a natural family $(f_\lambda)_{\lambda\in M}$ is also a holomorphic family, with the additional feature that the singular values of $f_\lambda$ \textit{and} their pre-images move holomorphically with $\lambda$. Astorg, Benini, and Fagella showed \cite[Theorem 2.6]{ABF22} that the converse also holds locally for finite-type maps: every holomorphic family of finite-type maps for which singular values and their pre-images move holomorphically can be locally expressed as a natural family.

We can now discuss structural stability in natural families of entire functions. The ``standard'' definition is the following:

\begin{definition}\label{def:Jstable}
Let $(f_\lambda)_{\lambda\in M}$ be a natural family of entire functions, where $M$ is a Banach analytic manifold. We say that \emph{$J(f_\lambda)$ moves holomorphically over $\Lambda\subset M$} (or, if we wish to be more precise, moves holomorphically \emph{in $(f_\lambda)_{\lambda\in \Lambda}$}) if, for every point $\lambda_0\in\Lambda$, there exists a neighbourhood $\Lambda'\subset\Lambda$ of $\lambda_0$ and a holomorphic motion of $J(f_{\lambda_0})$ over $\Lambda'$ such that:
\begin{itemize}
    \item For each $\lambda\in\Lambda'$, $H(\lambda, J(f_{\lambda_0})) = J(f_\lambda)$.
    \item $H(\lambda, f_{\lambda_0}(z)) = f_\lambda(H(\lambda, z))$ for every $(\lambda, z)\in \Lambda'\times J(f_{\lambda_0})$.
\end{itemize}
In other words, $H(\lambda, \cdot)$ conjugates $f_{\lambda_0}|_{J(f_{\lambda_0})}$ to $f_\lambda|_{J(f_\lambda)}$.
\end{definition}

\subsection{Universal natural families} 
In this subsection, we prove Theorem \ref{thm:unf}. As before, denote by $M(X)$, $X\subset \Cx$, the unit ball in the Banach space $L^\infty(X)\subset L^\infty(\Cx)$. Also, denote by $S(f)$ the set of singular values of $f$. We will prove Theorem \ref{thm:unf} as a consequence of the following more general result.
\begin{theorem}\label{thm:unfgeneral}
    Let $f$ be an entire function, and assume that $f$ has at least two singular values. Let $F = S(f)\cup\{\infty\}$, where $S(f)$ is the set of singular values of $f$. Then, there exists a Banach analytic manifold $T_f = T(F)\times(\Cx^*\times\Cx)^2$ and a surjective map $\Phi:T_f\to M_f\subset E$ satisfying the following properties:
    \begin{enumerate}[(i)]
        \item $\Phi$ is continuous relative to the topology of locally uniform convergence and the function $T_f\times\Cx\ni (\lambda, z)\mapsto \Phi_\lambda(z)\in\Cx$ is analytic.
        \item Let $(f_\lambda)_{\lambda\in M}$, where $M$ is a Banach analytic manifold, be a natural family containing  $f$. Then, there exists a holomorphic function $\phi:M\to T_f$ such that $f_\lambda = \Phi_{\phi(\lambda)}$. In other words, $\phi$ lifts the natural inclusion $M\ni \lambda\mapsto f_\lambda\in E$.
        \item The map $\Phi$ has unique path lifting. More specifically, given any continuous path $[0, 1]\ni t\mapsto g_t\in M_f$ and any point $\lambda_0\in \Phi^{-1}(g_0)$, there exists a unique path $\gamma\colon[0, 1]\to T_f$ such that $\gamma(0) = \lambda_0$ and $\Phi_{\gamma(t)} = g_t$ for all $t\in[0, 1]$.
    \end{enumerate}
    Furthermore, $T_f$ is finite-dimensional if and only if $f$ has finitely many singular values.
    \end{theorem}
Most of this subsection is devoted to the proof of Theorem \ref{thm:unfgeneral}. We start with a much easier result:
\begin{lemma} \label{lem:phiM}
For any entire function $f\in E$, there exists a function $\tilde\Phi:M(\Cx)\times(\Cx^*\times\Cx)^2\to M_f$ such that $\tilde\Phi$ is continuous, surjective, and the mapping $M(\Cx)\times(\Cx^*\times\Cx)^2\times \Cx \ni (\lambda, z)\mapsto \tilde\Phi_\lambda(z)\in\Cx$ is analytic.
\end{lemma}
\begin{proof}
Fix a function $f\in E$, and fix $z_0\in f^{-1}(0)$ and $z_1\in f^{-1}(1)$ (we can assume without loss of generality that neither $0$ nor $1$ is an ommited value by conjugating $f$ by an affine map). Let $\lambda = (\mu, a_1, b_1, a_2, b_2) \in M(\Cx)\times(\Cx^*\times\Cx)^2$. We start by applying Theorem \ref{lem:MRMT} to $\mu$, obtaining a quasiconformal map $\psi^\mu$ fixing $0$ and $1$. If we let $\psi_\lambda(z) = a_1\psi^\mu(z) + b_1$, then $\psi_\lambda$ ranges over all quasiconformal maps with Beltrami coefficient $\mu$ and is analytic in $\lambda$ by Theorems \ref{lem:MRMT} and \ref{lem:hartogs}. Notice that, because $\psi_\lambda(0)=b_1$ and $\psi_\lambda(1)=a_1+b_1$, the role of $a_1\in\Cx^*$, $b_1\in\Cx$ is that they determine normalisations for $\psi_\lambda$.

Next, let $\mu_\lambda = (\psi_\lambda\circ f)^*\mu_0 = f^*\mu$. It can be shown (see, for instance, \cite[p. 17]{BF14}) that
\[ \mu_\lambda(z) = \mu\left(f(z)\right)\frac{\overline{f'(z)}}{f'(z)}, \]
meaning that the map $\lambda\mapsto \mu_\lambda(z)$ is still analytic for every $z\in\Cx$. It is time to apply Theorem \ref{lem:MRMT} again to obtain a quasiconformal map $\varphi^{\mu_\lambda}$ with Beltrami coefficient $\mu_\lambda$ and fixing $z_0$ and $z_1$. Defining $\varphi_\lambda(z) = a_2\varphi^{\mu_\lambda}(z) + b_2$, this once again varies over all quasiconformal maps with Beltrami coefficient $\mu_\lambda$ as $a_2\in\Cx^*$ and $b_2\in\Cx$ vary. Once again, since $\varphi_\lambda(z_i)=a_2 z_i+b_2$, the variables $a_2$ and $b_2$ determine normalisations for $\varphi_\lambda$.

Defining the function $g_\lambda(z) = \psi_\lambda\circ f\circ(\varphi_\lambda)^{-1}(z)$, it follows from the construction that $g_\lambda$ is entire, and we see by an argument originally by Buff and Ch\'eritat \cite[p. 21]{BC04} that the map $\lambda\mapsto g_\lambda(z)$ is holomorphic in $\lambda$ for any fixed $z\in\Cx$. The continuity of $\lambda\mapsto g_\lambda$ follows from Theorem \ref{lem:MRMT}, and it is analytic as a map $(\lambda, z)\mapsto g_\lambda(z)$ by Hartogs' Theorem. Clearly, the function $\tilde\Phi$ is also surjective.
\end{proof}

Notice that we can already deduce Corollary \ref{cor:connected}:
\begin{proof}[Proof of Corollary \ref{cor:connected}]
For any entire function $f$, its equivalence class $M_f$ corresponds to the image under the continuous function $\tilde \Phi$ of the path-connected Banach manifold $M(\Cx)\times(\Cx^*\times\Cx)^2$, and is therefore path-connected.
\end{proof}

The domain $M(\Cx)\times(\Cx^*\times\Cx)^2$ is, in general, too big for our purposes -- because $M(\Cx)$ is too big. From now on, we assume that $0$ and $1$ are singular values of $f$ such that $\{0, 1\}\subset\partial S(f)$ (which, again, can always be achieved by conjugating $f$). It turns out, then, that we can push $\tilde\Phi$ down to a smaller space:
\begin{lemma} \label{lem:push}
Let $\tilde\Phi$ be as given by Lemma \ref{lem:phiM}, let $F = S(f)\cup\{\infty\}$, and let $\P_F:M(\Cx)\times(\Cx^*\times\Cx)^2\to T(F)\times(\Cx^*\times\Cx)^2$ be given by $\P_{F}(\mu, a_1, b_1, a_2, b_2) = (P_F(\mu), a_1, b_1, a_2, b_2)$ where $P_F$ is the map from Lemma~\ref{lem:TE}. Then, if $\lambda_1, \lambda_2\in M(\Cx)\times(\Cx^*\times\Cx)^{2}$ are such that $\P_F(\lambda_1) = \P_F(\lambda_2)$, we have $\tilde\Phi(\lambda_1) = \tilde\Phi(\lambda_2)$.
\end{lemma}
\begin{proof}
Let $\lambda_i = (\mu_i, a_{i,1}, b_{i,1}, a_{i,2}, b_{i,2})$, $i = 0, 1$. It is clear from the definition of $\P_F$ that we have $a_{0,j} = a_{1,j}$ and $b_{0,j} = b_{1,j}$ for $j = 1, 2$, so that we must focus on the role of $\mu_i$. We now consider two separate cases: first, if $S(f) = \Cx$, then (by Lemma \ref{lem:TE}) we have $P_F(\mu_0) = P_F(\mu_1)$ if and only if $\mu_0$ and $\mu_1$ have the same Beltrami coefficient. By uniqueness of the integrating quasiconformal map, it follows easily that $g_0 = g_1$.

Let us assume now that $S(f)$ is a proper subset of $\Cx$; the proof is essentially the same as that of \cite[Lemma 2]{EL92} with minor modifications. More specifically, let $g_0 = \tilde\Phi(\lambda_0) = \psi_0\circ f\circ(\varphi_0)^{-1}$ and $g_1 = \tilde\Phi(\lambda_1) = \psi_1\circ f\circ (\varphi_1)^{-1}$. Then, one has from the definition of $T(F)$ that $\psi_0$ and $\psi_1$ are $F$-equivalent, meaning that there exists an isotopy $\tilde\psi_t:\Cx\to\Cx$, $t\in[0, 1]$, such that $\tilde\psi_0 = \psi_0$, $\tilde\psi_1 = \psi_1$, and $\psi_t|_F$ is the identity for every $t\in[0, 1]$. Since $\Cx\setminus S(f)$ is non-empty, one can apply the covering homotopy theorem to find an isotopy $\tilde\varphi_t:\Cx\setminus f^{-1}(S(f))\to\Cx\setminus g_0^{-1}\circ\psi_0(S(f))$, $t\in[0, 1]$, such that $\tilde\varphi_0 = \varphi_0$, and $g_0\circ\tilde\varphi_t(z) = \tilde\psi_t\circ f(z)$ for all $z\in\Cx\setminus f^{-1}(S(f))$ and $t\in[0, 1]$. Moreover, since the isotopy $\tilde\psi_t$ is constant on $S(f)$, we can extend its lift $\tilde\varphi_t$ continuously to $f^{-1}(S(f))$ by defining $\tilde\varphi_t(z) = \varphi_0(z)$ for any $z\in f^{-1}(S(f))$. Setting $t = 1$, we obtain
\[ g_0\circ\tilde\varphi_1(z) = \psi_1\circ f(z) = g_1\circ \varphi_1(z), z\in\Cx\setminus f^{-1}(S(f)), \]
where the last equality follows from the definition of $g_1$. We see that $g_0 = g_1\circ(\tilde\varphi_1\circ\varphi_1^{-1})$. To proceed, recall that the lift $\tilde\varphi_t$ is constant on $f^{-1}(S(f))$ (which exact permutation of this set is enacted by $\tilde\varphi_t$ does not depend on $t$, and is uniquely determined by $a_{0,2}$ and $b_{0,2}$). Furthermore, since $g_0$ and $g_1$ are entire, the composition $\tilde\varphi_1\circ\varphi_1^{-1}$ is conformal on $\Cx\setminus f^{-1}(F)$, and hence has Beltrami coefficient zero on this set. If $S(f)$ has Lebesgue measure zero, we can apply Weyl's lemma to conclude that $\tilde\varphi_1$ and $\varphi_1$ differ by an automorphism of $\Cx$, and because $a_{0,2} = a_{1,2}$ and $b_{0,2} = b_{1,2}$ we have that this automorphism is the identity. We are done. If $S(f)$ has positive measure, we appeal to Lemma \ref{lem:TE}: $\psi_0$ and $\psi_1$, and in fact all the functions $\psi_t$, $t\in[0, 1]$, have the same Beltrami coefficients in $F$, and therefore so do $\tilde\varphi_1$ and $\varphi_1$. Thus, once again, $\tilde\varphi_1\circ\varphi_1^{-1}$ has zero Beltrami coefficient for $z\in f^{-1}(F)$. We can now apply Weyl's lemma, and the conclusion follows by the same argument as before.
\end{proof}
\begin{remark}
There is another way of passing from $M(\Cx)\times(\Cx^*\times\Cx)^2$ to $T_f$ that also works for functions with less than two singular values. Namely, when considering the set $F$, we assume that $0$ and $1$ are \textit{not} singular values of $f$, and take $F := \{0, 1, \infty\}\cup S(f)$. Doing it like this, however, gives us two ``extra'' dimensions of fibre for $\tilde\Phi$, corresponding to holomorphic motions of $F$ that move $0$ and $1$ but not the singular values of $f$. One then obtains the final manifold $T_f$ by projecting down to $T(F)\times(\Cx^*\times\Cx)^2$ and then projecting down again to remove the ``extra'' fibre.
\end{remark}

Lemma \ref{lem:push} shows that the projection $\mathcal{P}_F$ pushes $\tilde\Phi$ down to a well-defined function $\Phi\colon T(F)\times(\Cx^*\times\Cx)^2\to M_f$, which inherits all the properties of $\tilde\Phi$. This completes the proof of Theorem \ref{thm:unfgeneral}(i); we will see (ii) is now an easy consequence of Theorem \ref{lem:mitra}. Indeed, if $f_\lambda = \psi_\lambda\circ f\circ\varphi_\lambda^{-1}$, $\lambda\in M$, is a natural family containing $f$, then $\psi_\lambda$ induces a holomorphic motion of $F = S(f)\cup\{\infty\}$ that, by definition, admits a quasiconformal trivialization. Thus, by Theorem \ref{lem:mitra}, one can find a holomorphic function $\phi\colon M\to T_f$ (the other coordinates of $\phi$ can be ``read'' by tracking the images under $\psi_\lambda$ of $0$ and $1$, and under $\varphi_\lambda$ of the pre-images $z_0$ and $z_1$), giving rise to a natural family $\Phi_{\phi(\lambda)}$. The fact that the two natural families coincide follows from the fact that any two extensions of a holomorphic motion of $F$ are isotopic rel $F$ (see \cite[Theorem C]{Mit00}) by applying Lemma \ref{lem:push}.

Next, we prove Theorem \ref{thm:unfgeneral}(iii) as a consequence of the following lemma:
\begin{lemma}\label{lem:isotopy}
Let $g$ be an entire function, and let $\psi_t$ and $\varphi_t$, $t\in[0, 1]$, be quasiconformal homeomorphisms of $\Cx$ depending continuously on $t$ such that $\psi_t\circ g\circ\varphi_t^{-1} = g$ for every $t\in[0, 1]$ and $\psi_0(z) = z$ for all $z\in\Cx$. Then, $\psi_t$ is isotopic to the identity rel $F = S(f)\cup\{\infty\}$ for all $t\in[0, 1]$. In particular, $t\mapsto\psi_t(z)$ is constant for $z\in S(g)$, and $t\mapsto\varphi_t(z)$ is constant for $z\in g^{-1}(S(g))$.
\end{lemma}
\begin{proof}
Because $\psi_t$ moves continuously with $t$ and $\psi_0$ is the identity, it suffices to show that the map $t\mapsto \psi_t(z)$ is constant for $z\in S(g)$. To this end, notice first that (since $\psi_t$ and $\varphi_t$ are homeomorphisms of $\Chat$) $\psi_t(CV(g)) = CV(g)$, where $CV(g)$ denotes the critical values of $g$, and $\psi_t(AV(g)) = AV(g)$, where $AV(g)$ denotes the asymptotic values of $g$. Likewise, $\varphi_t(Crit(g)) = Crit(g)$, where $Crit(g)$ denotes the critical points of $g$. Since critical points are discrete in $\Cx$ and $\varphi_t$ moves continuously with $t$, it follows that $\varphi_t|_{Crit(g)}$ is constant in $t$, and therefore $\psi_t|_{CV(g)} = \psi_0|_{CV(g)} = Id$ for every $t\in[0, 1]$.

For dealing with the asymptotic values, we will need to look closer at ``preimages'' of asymptotic values. Let $v\in AV(g)$, and take $V\subset\Cx$ a small neighbourhood of $v$. If there are only finitely many asymptotic values of $g$ in $V$, then since $\psi_t$ moves continuously with $t$ and sends $AV(g)$ into $AV(g)$ we conclude that $\psi_t|_{AV(g)\cap V} = Id$ for all $t\in[0, 1]$. Thus, we can assume that $V$ contains infinitely many singular values of $g$. Since $g$ is entire, $g^{-1}(V)$ consists of at most countably many connected components, say $U_n$, and they are all simply connected. By the definition of an asymptotic value, each asymptotic value $v_\alpha\in V$ gives rise to an asymptotic path $\gamma_\alpha$ contained in one of the $U_n$, meaning that each $v_\alpha$ gives rise to at least one corresponding access to infinity in at least one of the components $U_n$. Again because $g$ is entire, each connected component of $\Cx\setminus U_n$ must contain at least one component of $\Cx\setminus V$; since there are at most countably many of those, it follows that each $U_n$ has at most countably many accesses to infinity.

Now, if $\psi_t|_{AV(g)\cap V}$ is not the identity, then $\varphi_t$ must ``undo'' the changes wrought by $\psi_t$. We claim that it can only do that by interchanging the corresponding accesses to infinity in the corresponding components $U_n$. To see that, fix any value of $n$, let $\phi\colon\D\to U_n$ be a Riemann map, and let $\gamma\subset U_n$ be an asymptotic path corresponding to $v$ (i.e., a curve $\gamma\colon[0, 1)\to U_n$ such that $\gamma(s)\to\infty$ and $g\circ\gamma(s)\to v$ as $s\nearrow 1$). The holomorphic function $\tilde g\colon\D\to\Cx$ given by $g\circ\phi$ satisfies $\tilde g\circ \tilde\gamma(s)\to v$ as $s\nearrow 1$, where $\tilde\gamma = \phi^{-1}(\gamma)$. By a classical theorem of Lindelöf \cite[Theorem I.2.2]{CG93}, $\tilde\gamma$ satisfies $\tilde\gamma(s)\to \zeta\in\partial\D$ as $s\nearrow 1$ for some $\zeta\in\partial\D$. Another classical theorem of Lehto and Virtanen \cite[Section 4.1]{Pom92} states that if any holomorphic function $\tilde g\colon\D\to\Cx$ omitting at least three points and satisfies $\tilde g\circ \alpha(s)\to c\in\Cx$ as $s\nearrow 1$, where $\alpha\colon[0, 1]\to\D$ is a curve landing at some point $e^{i\theta}\in\partial\D$, then $\tilde g$ has what is called a \textit{non-tangential limit} $c$ at $e^{i\theta}$; this means that $\tilde g$ tends to $c$ along any curve $\beta$ landing at $e^{i\theta}$ not tangentially to the unit circle. In particular, $\tilde g$ has a non-tangential limit $v$ at $\zeta$, and we can assume without loss of generality that $\gamma$ is chosen so that $\tilde\gamma$ lands non-tangentially at $\zeta$. Now, since $\varphi_t$ are quasiconformal maps depending continuously on $t$, the curves $\gamma_t = \varphi_t\circ\gamma$ must all define the same access to infinity in $U_n$, meaning that the curves $\tilde\gamma_t := \phi^{-1}(\gamma_t)$ all land on the same point $\zeta'\in\partial\D$ by the Correspondence Theorem (for more on accesses to infinity and the Correspondence Theorem, see \cite{BFJK17}). Moreover, since the maps $\tilde\varphi_t = \phi\circ\varphi_t\circ\phi^{-1}$ are quasiconformal, the curves $\tilde\gamma_t$ must land non-tangentially at $\zeta_1$. But by Lehto and Virtanen's theorem, $\tilde g$ has a non-tangential limit $v$ at $\zeta_1$ too (because $\psi_0$ is the identity, so $\tilde g\circ\tilde\gamma_0(s)\to v$ as $s\nearrow 1$); it follows that $\tilde g\circ\tilde\gamma_t(s)\to v$ as $s\nearrow 1$ for all $t\in[0, 1]$, and the same is true for $g$ and $\gamma_t$.

This proves our claim that $\varphi_t$ can only undo the changes made by $\psi_t$ by exchanging accesses to infinity. However, because this must be done continuously in $t$, it follows that $\varphi_t$ is incapable of enacting this permutation. Thus, $\psi_t|_{AV(g)}$ is the identity. Finally, since $S(g) = \overline{\{CV(g)\cup  AV(g)\}}$, the lemma follows by continuity.
\end{proof}
Theorem \ref{thm:unfgeneral}(iii) can be obtained from Lemma \ref{lem:isotopy} as follows. By the proof of Lemma~\ref{lem:phiM} we can write  $g_t=\psi_t\circ f\circ\varphi_t^{-1}$, but this lemma does not guarantee that $\psi_t$ and $\varphi_t$ can be chosen to be continuous in $t$. To show this, let $\lambda_0\in T_f$ be such that $\Phi_{\lambda_0} = g_0 = \psi_0\circ f\circ\varphi_0^{-1}$. Since each $g_t$ is quasiconformally equivalent to $f$, we claim that this induces a \textit{unique} continuous path $\gamma\colon[0, 1]\to T_f$ such that $\gamma(0) = \lambda_0$ and $\Phi_{\gamma(t)} = g_t$. Indeed, by continuity of the singular values \cite[Lemma 1]{KK97}, the path $t\mapsto g_t$ induces a continuous motion $H\colon [0, 1]\times F\to\Chat$ of $F = S(g_0)\cup\{\infty\}$, which, since each $S(g_t) = H(t, S(g_0))$ is the image of $S(g_0)$ under \textit{some} quasiconformal self-map of $\Cx$ by hypothesis, yields a unique continuous path $\tilde\gamma\colon[0, 1]\to T(F)\times\Cx^*\times\Cx$. Likewise, by the implicit function theorem, there exist unique continuous functions $t\mapsto a_2(t)$ and $t\mapsto b_2(t)$ tracking the pre-images of $H(t, 0)$ and $H(t, 1)$ under $g_t$ with $a_2(0) = \varphi_0(z_0)$ and $b_2(0) = \varphi_0(1)$. This, finally, yields a continuous path $\gamma\colon[0, 1]\to T_f$ with the desired properties, and guarantees the continuity of $\psi_t$ and $\varphi_t$ with respect to $t$.

Now, we prove uniqueness of this path. If $\gamma'\colon[0, 1]\to T_f$ is another path satisfying the required conditions, then we can write $g_t$ as $g_t = \Phi_{\gamma'(t)} = \psi'_t\circ g_0\circ(\varphi'_t)^{-1}$, where $\psi'_t$ and $\varphi'_t$ are the maps constructed in Lemma \ref{lem:phiM}. Combining this with the path $\gamma$ that we already have, and writing $g_t = \Phi_{\gamma(t)} = \psi_t\circ g_0\circ \varphi_t^{-1}$, we get $g_0 = (\psi_t^{-1}\circ\psi_t')\circ g_0\circ(\varphi_t\circ\varphi_t')^{-1}$, $t\in[0, 1]$. By Lemma \ref{lem:isotopy}, $\psi_t$ and $\psi_t'$ are isotopic relative $F$ for all $t\in[0, 1]$. In particular, both define the same continuous motion of $F$, and so by \cite[Lemma 12.2]{Mit00} the lift to $T_f$ is unique.

This completes the proof of Theorem \ref{thm:unfgeneral}.
Next, let us deduce Theorem \ref{thm:unf}. Since $T_f$ is a manifold and $M_f$ is a Hausdorff space, we have that $\Phi$ is a covering map if and only if it is locally injective and has the path lifting property (see \cite[Theorem 4.19]{For81}). Therefore, the fact that $\Phi$ is a covering will be proved once we show the following:
\begin{lemma}
Assume that $\Cx\setminus S(f)$ is non-empty. Then, for every $\lambda = (P_F(\mu), a_1, b_1, a_2, b_2)\in T_f$, there exists a neighbourhood $\Lambda\subset T_f$ of $\lambda$ such that $\Phi|_\Lambda$ is injective.
\end{lemma}
\begin{proof}
First, notice that for any $g = \psi_g\circ f\circ\varphi_g^{-1}\in M_f$, a pair of quasiconformal homeomorphisms $\psi$ and $\varphi$ such that $\psi\circ g\circ\varphi^{-1} = g$ gives rise to a pair $\psi' = \psi_g^{-1}\circ\psi\circ\psi_g$ and $\varphi' = \varphi_g^{-1}\circ\varphi\circ\varphi_g$ such that $\psi'\circ f\circ(\varphi')^{-1}$, and vice-versa. Thus, we can assume without loss of generality that $g = f$.

With that out of the way, assume that our claim is false; then, there exists a sequence $(\lambda_n)_n\subset T_f$ converging to the basepoint $(0, 1, 0, 1, 0)$ such that $\Phi(\lambda_n) = f$. By the construction of $\Phi$, this gives rise to sequences $\psi_n$ and $\varphi_n$ of quasiconformal maps of $\Cx$ such that $\psi_n\circ f\circ \varphi_n^{-1} = f$, $\psi_n$ and $\varphi_n$ converge locally uniformly to the identity as $n\to+\infty$, and $\psi_n$ is not isotopic to the identity relative to $F = S(f)\cup\{\infty\}$ for any $n$. However, we claim that, for any sufficiently large $n$, one can construct a family of quasiconformal maps $\psi_{n,t}$ depending continuously on $t\in[0, 1]$, with $\psi_{n,1}$ isotopic to $\psi_n$ relative $F$, $\psi_{n,0}(z) = z$ for all $z\in\Cx$, and such that $\psi_{n,t}\circ f\circ\varphi_{n,t}^{-1} = f$ for all $t\in[0, 1]$. If we can do that, then by Lemma \ref{lem:isotopy}, $\psi_n$ is isotopic to the identity relative $F$, which is a contradiction -- completing the proof.

Now, to construct this family of quasiconformal maps. Let $D\subset\Cx$ be a (round) disc compactly contained in $\Cx\setminus S(f)$. Then, for all large $n$, $\psi_n(\Cx\setminus D)$ does not intersect $D'$, where $D'\Subset D$ is another (round) disc that does not depend on $n$. We perform quasiconformal interpolation in the annulus $\overline{D}\setminus D'$ (see e.g. \cite[Proposition 2.32]{BF14}), obtaining a quasiconformal map $\psi_{n,1}$ such that $\psi_{n,1}(z) = \psi_n(z)$ for $z\in\Cx\setminus D$ and $\psi_{n,1}(z) = z$ for $z\in\overline{D'}$. We define $\varphi_{n,1}$ to be an integrating map of $f^*\mu_{n,1}$, where $\mu_{n,1}$ is the Beltrami coefficient of $\psi_{n,1}$, normalised so that $\varphi_{n,1}(z_0) = \varphi_n(z_0)$ and $\varphi_{n,1}(z_1) = \varphi_n(z_1)$ (recall the definition of $z_0$ and $z_1$ from the proof of Lemma \ref{lem:phiM}). Because we have only made small changes outside of $S(f)$, we have by Lemma \ref{lem:push} that $\psi_{n,1}\circ f\circ\varphi_{n,1}^{-1} = f$ for all large $n$.

Next, notice that, since $\overline{D}$ is contained in $\Cx\setminus S(f)$, its preimage consists of countably many compact sets; in particular, we can find a compact set $D''\subset f^{-1}(D)$ such that $f$ maps $D''$ conformally onto $D'$. We claim that, for all large $n$, $\varphi_{n,1}$ is the identity on $D''$; indeed, because $\psi_{n,1}\circ f\circ\varphi_{n,1}^{-1} = f$, the map $\varphi_{n,1}$ must permute the components of $f^{-1}(D')$, and since $\varphi_n(z)\to z$ locally uniformly as $n\to+\infty$ it follows that for all large $n$ we must have $\varphi_{n,1}(D'') \subset D''$. Hence, for $z\in D''$, we have $f(z) = f\circ\varphi_{n,1}(z)$, as $\psi_{n,1}(z)$ for $z\in D' = f(D'')$, and so $\varphi_{n,1}(z) = z$ for $z\in D''$.

Finally, let $\psi_{n,t}$, $t\in[0, 1]$, be the integrating map of $\mu_{n,t} := t\mu_{n,1}$ (where, recall, $\mu_{n,1}$ is the Beltrami coefficient of $\psi_{n,1}$) normalised so that $\psi_{n,t}(0) = \psi_n(0)$ and $\psi_{n,t}(1) = \psi_n(1)$. Likewise, take $\varphi_{n,t}$ to be the integrating map of $f^*(\mu_{n,t})$ such that $\varphi_{n,t}(z_0) = \varphi_n(z_0)$ and $\varphi_{n,t}(z_1) = \varphi_n(z_1)$. Furthermore, by using quasiconformal surgery and the fact that the integrating map varies continuously with its Beltrami coefficient, we can modify $\psi_{n,t}$ and $\varphi_{n,t}$ so that they are the identity when restricted to $D'$ and $D''$ (respectively), and so $f_t(z) := \psi_{n,t}\circ f\circ\varphi_{n,t}^{-1}(z) = f(z)$ for $z\in D''$; by the identity principle, $f_t = f$ for all $t\in[0, 1]$. It is clear from the construction that $\psi_{n,t}$ and $\varphi_{n,t}$ satisfy the other properties outlined above, and so the proof is complete.
\end{proof}

Finally, we complete the proof of Theorem \ref{thm:unf}.
\begin{proof}[Proof of Theorem \ref{thm:unf}]
We have already shown that $\Phi\colon T_f\to M_f$ is a covering, and properties (i) and (ii) in Theorem \ref{thm:unf} follow from the corresponding properties in Theorem \ref{thm:unfgeneral}. Thus, we only have to give $M_f$ its complex structure -- in other words, a manifold structure with holomorphic transition maps.

To that end, let $g\in M_f$. Since $\Phi$ is a covering, there exists a neighbourhood $V\subset M_f$ of $g$ such that $\Phi^{-1}(V)$ is a disjoint union of neighbourhoods $U_n\subset T_f$, and $\Phi|_{U_n}$ is a homeomorphism for each $n$. Let $\psi_n$ denote the inverse of the restriction $\Phi|_{U_n}\colon U_n\to V$; these will be our charts. To see that the transition charts are holomorphic, notice that they are given by $\psi_n\circ\Phi|_{U_m}\colon U_m\to U_n$, and can be written ``component-wise'' as $U_m\ni \lambda \mapsto (\mu(\lambda), a_1(\lambda), b_1(\lambda), a_2(\lambda), b_2(\lambda))\in U_n$. The first three components, $\mu(\lambda)$, $a_1(\lambda)$, and $b_1(\lambda)$, are holomorphic functions of $\lambda$ by Theorem \ref{lem:mitra}, since $\Phi$ defines a holomorphic motion $(\lambda, z)\mapsto \Phi_{\lambda}(z)$ of $S(f)$. The last two components, $a_2(\lambda)$ and $b_2(\lambda)$, are holomorphic by the argument of Buff and Ch\'eritat \cite[p. 21]{BC04}, which shows that the map $\lambda\mapsto \varphi_\lambda(z)$ is holomorphic for $z\in f^{-1}(S(f))$. By Hartogs' Theorem (Theorem \ref{lem:hartogs}), the transition map $\lambda\mapsto (\mu(\lambda), a_1(\lambda), b_1(\lambda), a_2(\lambda), b_2(\lambda))$ is holomorphic, and thus the atlas given here defines a complex structure for $M_f$.
\end{proof}

We end this section with a proof of Corollary \ref{cor:natfamily}, which is an immediate consequence of Theorem \ref{thm:unf}, and whose proof is similar to the argument above.
\begin{proof}[Proof of Corollary \ref{cor:natfamily}]
First, the fact that $\psi_\lambda(z)$ is already defined on $\Cx$ and moves holomorphically with $\lambda$ for $z\in S(f)$ defines a holomorphic function $\Psi\colon M\to T(F)\times(\Cx^*\times\Cx)$ (recall Theorem \ref{lem:mitra}); to extend this to a holomorphic function $F\colon M\to T_f = T(F)\times(\Cx^*\times\Cx)^2$, recall that by the usual argument of Buff and Ch\'eritat \cite[p. 21]{BC04} the map $\lambda\mapsto \varphi_\lambda(z)$ is holomorphic for $z\in f^{-1}(S(f))$. We can thus define $f_\lambda' = \Phi_{F(\lambda)} = \psi_\lambda'\circ f\circ (\varphi_\lambda')^{-1}$, and we are left to show that $f_\lambda' = f_\lambda$. This, however, follows immediately from the fact that $\psi_\lambda$ and $\psi_\lambda'$ are two different quasiconformal trivializations of the same holomorphic motion of $F = S(f)\cup\{\infty\}$, and thus isotopic rel $F$ by \cite[Theorem C]{Mit00}; by Lemma \ref{lem:push}, $f_\lambda = f_\lambda'$.
\end{proof}

\section{Distortion sequences for wandering domains}\label{sec:DS}
In this section, we define distortion sequences for entire functions with simply connected wandering domains and prove Theorem \ref{thm:DS}.
\begin{definition}\label{def:DS}
Let $f$ be an entire function with a simply connected wandering domain $U$. Take $z_0\in U$, and let $z_n := f^n(z_0)$. For $n\geq 0$, let $\psi_n\colon\D\to U_n$ be Riemann maps with $\psi_n(0) = z_n$. For $n\in\N$, define $g_n\colon\D\to\D$ by $g_n(z) = \psi_n^{-1}\circ f\circ\psi_{n-1}(z)$, so that $g_n(0) = 0$. The sequence $(\alpha_n(f, z_0))_{n\in\N}$ given by
\[ \alpha_n(f, z_0) := g_n'(0) \]
is called a \emph{distortion sequence of $f$ at $z_0$}.
\end{definition}
\begin{remark}
We use the term ``distortion'' because, by the Schwarz--Pick lemma, each $\alpha_n(f, z)$ is related to a quantity called the \textit{hyperbolic distortion} of $f\colon U_{n-1}\to U_n$ at $f^{n-1}(z)$; see e.g. \cite{BM06} for the definition of hyperbolic distortion, or \cite{BEFRS19} for its use in this context. 
\end{remark}

Let us take a moment to discuss how the distortion sequence depends on our choice of $z\in U$. With the setup of Definition \ref{def:DS}, we can, for any choice of $z\in U$, obtain Riemann maps $\tilde\psi_n\colon U_n\to\D$ with the property that $\tilde\psi_n(f^n(z)) = 0$ by post-composing $\psi_n$ with a M\"obius transformation preserving $\D$ and mapping $\psi_n(f^n(z))$ to the origin. Because M\"obius self-maps of $\D$ move real-analytically with their root, we conclude that, for any $n\in\mathbb{N}$, the map $U\ni z\mapsto \alpha_n(f, z)\in\D$ is real-analytic. This does not mean, however, that the map $U\ni z\mapsto (\alpha_n(f, z))_{n\in\mathbb{N}}\in\ell^\infty$ is real-analytic; see the discussion after Lemma \ref{lem:linfty}. Nevertheless, since the distortion sequence of $f$ at $z$ is closely related to the hyperbolic distortion of $f$ at $f^n(z)$, we can apply the Schwarz--Pick lemma for hyperbolic distortion (see \cite[Theorem 11.2]{BM06}) and the fact that $f^n\colon U\to U_n$ does not expand the hyperbolic metric to conclude that $U\ni z\mapsto (\alpha_n(f, z))_{n\in\mathbb{N}}\in\ell^\infty$ is at the very least continuous.

Notice also that the distortion sequence of $f$ at $z_0$ is not unique, but any two distortion sequences $(\alpha_n)_{n\in\N}$ and $(\beta_n)_{n\in\N}$ at $z_0$ satisfy $|\alpha_n| = |\beta_n|$ for all $n\in\N$. By Theorem \ref{thmout:befrs}, a distortion sequence of $f$ at $z_0\in U$ is closely related to the internal dynamics of $U$. By the same token, distortion sequences at different points $z\in U$ must have the same ``qualitative behaviour'', meaning here that the series $\sum_{n\geq 1} (1 - |\alpha_n(f, z)|)$ behaves in the same way for every $z\in U$.

As discussed in Section \ref{sec:intro}, if $J(f_\lambda)$ moves holomorphically in some natural family $(f_\lambda)_{\lambda\in M}$ with holomorphic motion $H$, then $f_\lambda$ has a simply connected wandering domain $U_\lambda = H(\lambda, U)$, and $U_{\lambda,n} = f_\lambda^n(U_\lambda) = H(\lambda, U_n)$. To understand the distortion sequences of $f_\lambda$ at $H(\lambda, z_0)$, we must understand the Riemann maps $\zeta_{\lambda,n}$ of $U_{\lambda,n}$. We have:
\begin{lemma}\label{lem:riemannhol}
Let $\Omega\subsetneq\Cx$ be a simply connected domain, and let $H\colon M\times\Omega\to\Cx$ be a holomorphic motion of $\Omega$ over some Banach analytic manifold $M$ with basepoint $\lambda_0\in M$. Let $p\in\Omega$, and let $\zeta\colon\D\to\Omega$ be a Riemann map with $\zeta(0) = p$. Then, there exists a holomorphic motion $\eta\colon M\times\overline{M}\times\D\to\Omega_{\lambda,\gamma}' := \eta(\lambda, \gamma)(\D)\subset\Cx$ such that:
\begin{enumerate}[(1)]
    \item For all $(\lambda, \gamma)\in M\times\overline{M}$, $\eta(\lambda, \gamma)(0) = 0$;
    \item For each $\lambda\in M$, $\eta(\lambda, \overline\lambda)(\D) = \D$; and
    \item For all $z\in\D$, $\eta(\lambda_0, \overline{\lambda_0})(z) = z$.
\end{enumerate}
Furthermore, for $(\lambda, \gamma)\in M\times\overline{M}$, let $Z(\lambda, \gamma)\colon\Omega_{\lambda, \gamma}'\to\Cx$ be given by $Z(\lambda, \gamma) = H(\lambda, \zeta\circ\eta(\lambda, \gamma)^{-1}(z))$. Then:
\begin{enumerate}[(i)]
    \item For every $(\lambda, \gamma)\in M\times\overline M$, $Z(\lambda, \gamma)\colon\Omega_{\lambda,\gamma}'\to\Omega_\lambda := H(\lambda, \Omega)$ is a biholomorphism;
    \item For each $\lambda\in M$, the function $\zeta_\lambda(z) := Z(\lambda, \overline\lambda)(z)$ is a Riemann map of $\Omega_\lambda$;
    \item For all $(\lambda, \gamma)\in M\times\overline M$, $Z(\lambda, \gamma)(0) = H(\lambda, p)$;
    \item For all $z\in\D$, $Z(\lambda, \overline{\lambda_0})(z) = \zeta(z)$; and
    \item For all $(\lambda^*, \gamma^*, z)\in M\times\overline M\times \Omega_{\lambda^*,\gamma^*}'$ and any neighbourhood $\Lambda\subset M\times\overline M$ of $(\lambda^*, \gamma^*)$ such that $z\in\Omega_{\lambda,\gamma}'$ for $(\lambda, \gamma)\in \Lambda$, the map $\Lambda\ni (\lambda, \gamma)\mapsto Z(\lambda, \gamma)(z)\in \Cx$ is holomorphic.
\end{enumerate}
\end{lemma}
\begin{proof}
Let $h_\lambda(z) := H(\lambda, z)$, so that $h_\lambda$ is quasiconformal by Lemma \ref{lem:lambda1}. We start by defining a Beltrami coefficient $\mu_\lambda\colon\D\to\D$ as $\mu_\lambda := (h_\lambda\circ\zeta)^*\mu_0$. It follows from Lemma \ref{lem:MRMTconv} that $\mu_\lambda(z)$ is analytic in $\lambda$ for each $z\in\D$, but, since we want to preserve the unit disc, we cannot integrate it directly to obtain an integrating map depending analytically on $\lambda$. To be more precise, we could set $\mu = 0$ outside of $\D$ and integrate the resulting Beltrami coefficient to obtain a quasiconformal map moving holomorphically with $\lambda$, but we could not guarantee that it maps $\D$ to $\D$.

To circumvent this problem, consider that the space $L^\infty(\Cx)$ can be split as $L^\infty(\Cx) = L^\infty(\D)\oplus L^\infty(\Cx\setminus\overline{\D})$. Theorem \ref{lem:MRMT} can therefore be restated as saying that the map
\[ L^\infty(\D)\times L^\infty(\Cx\setminus\overline\D)\times\Cx\ni (\mu, \nu, z)\mapsto \varphi^{\mu+\nu}(z)\in\Cx \]
is holomorphic. If, furthermore, we take $\nu = \tau^*\mu$, where $\tau(z) = 1/\bar z$, then the corresponding integrating map $\varphi^{\mu+\nu}$ is symmetric with respect to reflection across the unit circle, and therefore $\varphi^{\mu+\nu}(\D) = \D$ (see \cite[Exercise 1.4.1]{BF14} or \cite[Lemma 14]{AB60}). Its restriction to the unit disc is therefore a quasiconformal self-map of the unit disc, fixing the origin (and one) and integrating $\mu$.

Applying this idea to our situation, we let $\mu = \mu_\lambda = (h_\lambda\circ\zeta)^*\mu_0$ and $\nu = \nu_\gamma = \tau^*(h_\gamma\circ\zeta)^*\mu_0$, where $(\lambda, \gamma)\in M\times\overline M$. We obtain the map $\eta\colon M\times\overline M\times\D\to\Cx$ given by
\[ \eta(\lambda, \gamma)(z) := \varphi^{\mu_\lambda + \nu_\gamma}(z); \]
it follows from the definition that $\mu_\lambda$ and $\nu_\gamma$ are holomorphic in $\lambda$ and $\gamma$ (respectively), so that we conclude from Theorem \ref{lem:MRMT} that $\eta$ is a holomorphic motion. Properties (1), (2), and (3) are readily established from the definition of $\varphi^{\mu_\lambda + \nu_\gamma}$.

If we define $Z(\lambda, \gamma)\colon\Omega_{\lambda,\gamma}'\to\Cx$ as proposed, property (i) follows from the fact that -- by construction -- it preserves the standard almost complex structure $\mu_0$ (see Weyl's lemma, e.g. \cite[Theorem 1.14]{BF14}). Properties (ii), (iii), and (iv) now follow from (2), (1), and (3), respectively. Finally, property (v) follows by the usual argument of Buff and Ch\'eritat \cite[p. 21]{BC04}.
\end{proof}
\begin{remark}
The question of how Riemann maps relate to holomorphic motions of the domain is a deep one; other results on the topic can be found in \cite{PR86,Rod86,Zak16}.
\end{remark}
Before we can apply Lemma \ref{lem:riemannhol} to our setting, we need to extend our holomorphic motion of the Julia set to the whole orbit of the wandering domain $U$. Given Theorem \ref{lem:lambda2}, this is not an unsurmountable hurdle -- but the extension must also respect the dynamics of $f$ at $p\in U$, and this must be handled more carefully:
\begin{lemma}\label{lem:extend}
Let $f = f_{\lambda_0}\in (f_\lambda)_{\lambda\in M}$ be an entire function with a wandering domain $U$, and let $p\in U$. Let $H\colon\Lambda\times J(f)\to\Cx$ be a holomorphic motion with basepoint $\lambda_0$ conjugating $f|_{J(f)}$ to $f_\lambda|_{J(f_\lambda)}$, where $\Lambda\subset M$ is a neighbourhood of $\lambda_0$. Then, there exists a neighbourhood $\Lambda'\subset\Lambda$ of $\lambda_0$ and an extension $\tilde H\colon\Lambda'\times\Cx\to\Cx$ such that, for all $\lambda\in \Lambda'$ and $n\in\mathbb{N}_0$,
\[ \tilde H(\lambda, f^n(p)) = f_\lambda(\tilde H(\lambda, p)). \]
In other words, $\tilde H(\lambda, \cdot)$ conjugates the $f$-orbit of $p$ to the $f_\lambda$-orbit of $\tilde H(\lambda, p)$ for all $\lambda\in\Lambda'$.
\end{lemma}
\begin{proof}
Because any Banach analytic manifold is locally simply connected, we may assume without loss of generality that $\Lambda$ is simply connected. Then, Theorem \ref{lem:lambda2} gives us a smaller neighbourhood $\Lambda''\subset\Lambda$ and a unique extension $K\colon \Lambda''\times\Cx\to\Cx$ of $H$ with harmonic Beltrami coefficient. In particular, $\lambda\mapsto K(\lambda, p)$ moves holomorphically with $\lambda\in\Lambda''$, and so does $f_\lambda^n(K(\lambda, p))$ for every $n\in\mathbb{N}$. This allows us to define a new holomorphic motion $H_1$ of $J(f)\cup\{f^n(p)\}_{n\geq 0}$ over $\Lambda''$, which respects the dynamics, by setting
\[ H_1(\lambda, z) = \begin{cases} H(\lambda, z), & z\in J(f) \\
                        f_\lambda^n(K(\lambda, p)), & \text{$z = f^n(p)$ for some $n\geq 0$.} \end{cases} \]
We can once again apply Theorem \ref{lem:lambda2} to extend $H_1$ to a holomorphic motion $\tilde H$ of $\Cx$ over some neighbhourhood $\Lambda\subset\Lambda''$. By construction, $\tilde H$ respects the dynamics on $J(f)$ and on the orbit of $p\in U$.
\end{proof}
\begin{remark}
The argument used here to extend the holomorphic motion from $J(f)$ into the Fatou set while also preserving the dynamics of a given orbit does not always work for pre-periodic Fatou components. That is because, in such a component, it cannot be ruled out that either $f^n(p)$ or $f_\lambda^n(K(\lambda, p))$ is periodic (for $f$ or $f_\lambda$, respectively). In that case, the new ``holomorphic motion'' $H_1$ is either not well-defined or not injective.
\end{remark}
We can now prove Theorem \ref{thm:DS}.
\begin{proof}[Proof of Theorem \ref{thm:DS}]
We begin by taking Riemann maps $\zeta_n\colon\D\to U_n$, $n\geq 0$, with basepoints $p_n = f^n(p)$. Then, we apply Lemma \ref{lem:extend} to the dynamics-preserving holomorphic motion $H$ of $J(f)$ over $\Lambda\subset M$ (which exists by hypothesis), obtaining a holomorphic motion $\tilde H$ that preserves the dynamics on $J(f)\cup\{p_n\}_{n\geq 0}$. Applying Lemma \ref{lem:riemannhol}, we obtain holomorphic motions $\eta_n\colon \Lambda'\times\overline{\Lambda'}\times\D\to\Cx$, $n\geq 0$, such that $\zeta_{\lambda,n}(z) := \tilde H\left(\lambda, \zeta_n\circ\eta_n(\lambda,\overline\lambda)^{-1}(z)\right)$ are Riemann maps of $U_{\lambda,n} = f_\lambda^n(U_\lambda) = \tilde H(\lambda, U_n)$ normalised so that $\zeta_{\lambda,n}(0) = \tilde H(\lambda, p_n)$. The functions $g_{\lambda,n}\colon\D\to\D$ defined by
\[ g_{\lambda,n}(z) := \zeta_{\lambda,n}^{-1}\circ f_\lambda\circ\zeta_{\lambda,n-1}(z) \]
form a sequence of inner functions that is conjugate to $(f_\lambda|_{U_{\lambda,n}})_{n\in\N}$, but they are not holomorphic in $\lambda$. They can, however, be understood as the restriction to $(\lambda, \overline\lambda)\in \Lambda'\times\overline{\Lambda'}$ of the more general functions
\[ G(\lambda, \gamma)(z) := Z_n(\lambda, \gamma)^{-1}\circ f_\lambda\circ Z_{n-1}(\lambda, \gamma)(z), \]
where $Z_n(\lambda, \gamma)\colon \eta_n(\lambda,\gamma)(\D)\to\Cx$ are also given by Lemma \ref{lem:riemannhol}. These functions are holomorphic in $\lambda, \gamma$, and $z$; it follows that the maps $\alpha_n\colon \Lambda'\times\overline{\Lambda'}\to\Cx$ given by
\[ \alpha_n(\lambda, \gamma) := G_n(\lambda, \gamma)'(0) \]
are holomorphic. Furthermore, since the quasiconformal maps $(\eta_n(\lambda, \gamma))_{n\in\N}$ have dilatation uniformly bounded in terms of $\lambda$ and $\gamma$, it follows that the sequence $(\alpha_n(\lambda, \gamma))_{n\in\N}$ is uniformly bounded on compact subsets of $\Lambda'\times\overline{\Lambda'}$ by Cauchy's integral formula. Hence, the function $\A\colon \Lambda'\times\overline{\Lambda'}\to\ell^\infty$ given by
\[ \A(\lambda, \gamma) := \left(\alpha_n(\lambda, \gamma)\right)_{n\in\N} \]
is holomorphic by Lemma \ref{lem:linfty}. Finally, it follows from the construction of $\A$ that $\A(\lambda, \overline\lambda) = (g_{\lambda,n}'(0))_{n\in\N}$, completing the proof.
\end{proof}
\begin{remark}
Because extensions of holomorphic motions are not in general unique, a distortion map built in this manner is also not in general unique. However, since $g_{\lambda,n}'(0)$ is also the hyperbolic distortion of $f_\lambda$ at $\tilde H(\lambda, f^n(p))$, which is uniquely determined by the geometry of $U_{\lambda,n}$ and $f_\lambda$ itself, any two distortion maps $\A$ and $\A'$ agree on the modulus of $g_{\lambda,n}'(0)$ for all $n$.
\end{remark}
Corollary \ref{cor:real} is an immediate consequence of Theorem \ref{thm:DS}, since $\{(\lambda, \overline{\lambda})\colon\lambda\in\Lambda'\}$ is the graph of a real-analytic function. Corollary \ref{cor:contr} now follows immediately from Corollary \ref{cor:real} by noting that $\|\A(\lambda,\overline\lambda)\|_\infty$ varies continuously with $\lambda\in \Lambda'$, and thus if $\|\alpha\|_\infty < 1$ then there exists a neighbourhood $\Lambda''\subset M$ of $\lambda_0$ for which $\|\A(\lambda, \overline\lambda)\|_\infty < 1$, which implies that the corresponding wandering domain $U_\lambda$ is contracting by Theorem \ref{thmout:befrs}.

\section{Perturbing Herman-type wandering domains}\label{sec:Herman}
Theorem \ref{thm:Herman} follows immediately from the more precise statement.

\begin{theorem}\label{thm:Hermanlong}
Let $f\colon\Cx\to\Cx$ be a transcendental entire function with an attracting Herman-type wandering domain $U$, and let $M = \{\lambda = (\lambda_1, \lambda_2, \ldots)\in\ell^\infty : \|\lambda\|_\infty < 1\}$. Then, there exist $z_0\in U$ and a natural family $(f_\lambda)_{\lambda\in M}$ such that:
\begin{enumerate}[(i)]
    \item $f_0 = f$;
    \item $J(f_\lambda)$ moves holomorphically in $(f_\lambda)_{\lambda\in M}$;
    \item If $\alpha$ is a distortion sequence of $f$ at $z_0$ and $\alpha_n(\lambda)$ denotes the $n$-th entry of the distortion map $\A(\lambda, \overline\lambda)$ of $\alpha$ over $M$, then
    \[ \frac{d\alpha_n}{d\lambda_n}(0)\neq 0. \]
\end{enumerate}
\end{theorem}

The proof of Theorem \ref{thm:Hermanlong} occupies the rest of this section. We will use a quasiconformal surgery parametrised by the proposed Banach analytic manifold $M$.

Let $z_0\in U$ denote the lift in $U$ of the attracting fixed point $w_0$ of $h\colon\Cx^*\to\Cx^*$, where $\exp\circ f = h\circ\exp$, and let $z_n := f^n(z_0)$. Denoting the immediate basin of attraction of $w_0$ by $V$, let $\zeta\colon\D\to V$ be a Riemann map with $\zeta(0) = w_0$ and $\arg \zeta'(0) = \arg w_0$, and let $\exp_n^{-1}$ denote the branch of the logarithm on $V$ mapping $w_0$ to $z_n$ for $n\geq 0$. We see that $\zeta_n := \exp_n^{-1}\circ\zeta$ is a Riemann map of $U_n$ satisfying $\zeta_n(0) = z_n$ and $\zeta_n'(0) > 0$. It follows that we have an inner function $g\colon\D\to\D$ such that, for all $n\in\N$,
\[ g(z) = \zeta_n^{-1}\circ f\circ \zeta_{n-1}(z), \]
with $g(0) = 0$ and $\alpha := g'(0) = h'(w_0)$. In particular, $\alpha = (\alpha, \alpha, \ldots)_{n\in\N}$ is a distortion sequence for $f$ at $z_0$ (this is the only mention we will make of the distortion sequence $\alpha$, as opposed to the multiplier $\alpha$. We hope this will not cause confusion).

By Koenig's linearisation theorem \cite[Theorem 8.2]{Mil06}, there exist a neighbourhood $\Delta\subset\D$ of the origin and a biholomorphism $\psi\colon\Delta\to\D_L := \{z : |z| < L\}$ such that
\[ \psi\circ g(z) = \alpha\cdot\psi(z)\text{ for }z\in\Delta \]
and, furthermore, $\psi'(0) = 1$.

Now, we wish to use the linearised coordinates $w = \psi(z)$ to substitute the action $w\mapsto \alpha w$ of $g$ for the action $w\mapsto \tilde\alpha_nw$, where $\tilde\alpha_n := \alpha + \rho\lambda_n$ and $\rho := |\alpha|/2\cdot\min\{|\alpha, 1 - \alpha\}$. This will require the following kind of quasiconformal interpolation:
\begin{lemma}\label{lem:interpol}
Let $R > 0$, let $\alpha\in\D^*$, let $r = |\alpha|R$, and let $\rho = |\alpha|/2\cdot\min\{|\alpha|, 1 - |\alpha|\}$. Then, the map $\varphi\colon\D\times\{z : r < |z| < R\}\to\Cx$ given by
\[ \varphi(\lambda, te^{i\theta}) = \left(\frac{R - t}{R - r}\tilde\alpha + \frac{t - r}{R - r}\alpha\right)te^{i\theta}, \]
where $\tilde\alpha = \alpha + \rho\lambda$ and $\lambda\in\D$, satisfies the following properties.
\begin{enumerate}[(1)]
    \item For every $\lambda\in\D$, $\varphi(\lambda, \cdot)$ interpolates between $re^{i\theta}\mapsto \tilde\alpha re^{i\theta}$ and $Re^{i\theta}\mapsto \alpha Re^{i\theta}$.
    \item For every $\lambda\in\D$, the map $\varphi_\lambda := \varphi(\lambda, \cdot)$ is a quasiconformal homeomorphism.
    \item For every fixed $z$, the map $\D\ni\lambda\mapsto \mu(\lambda) := \varphi_\lambda^*\mu_0(z)$ is analytic and $|\mu(\lambda)| \leq |\lambda|$.
\end{enumerate}
\end{lemma}
\begin{proof}
Property (1) is clear by substituting $t = r$ and $t = R$ into the definition of $\varphi$. To prove (2), we start by noting that $\varphi_\lambda$ maps the circle of radius $t = r + s(R - r)$, $0\leq s\leq 1$, homeomorphically onto the circle of radius $\sigma(s) = |\alpha + (1 - s)\rho\lambda|(r + s(R - r))$. Because $\rho$ is small compared to $\alpha$, $\sigma(s)$ is injective, and hence $\varphi_\lambda$ is an orientation-preserving diffeomorphism between $\{z : r < |z| < R\}$ and $\{z : \tilde\alpha r < |z| < \alpha R\}$. It follows that it is quasiconformal, since it extends smoothly to the closure (see \cite[Remark 1.6(b)]{BF14}). Finally, it is clear that, for each $z = te^{i\theta}$, the map $\D\ni\lambda\mapsto \varphi_\lambda(z)$ is analytic, and thus by Lemma \ref{lem:MRMTconv} so is its Beltrami coefficient. Property (3) now follows from the Schwarz lemma \cite[Theorem 7.19]{Muj86}.
\end{proof}

We want to apply Lemma \ref{lem:interpol} to our situation: we choose $R\in (0, L)$ (we will see further ahead, in Claim \ref{cl:approx}, that $R$ must be ``small''), and interpolate between $w\mapsto\tilde\alpha_n w = (\alpha + \rho\lambda_n)w$ on $\{w\colon |w| = |\alpha| R\}$ and $w\mapsto \alpha w$ on $\{w\colon |w| = R\}$, obtaining the quasiconformal maps $\varphi_n(w) = \varphi(\lambda_n, w)$. Also by Lemma \ref{lem:interpol}, the Beltrami coefficients $\varphi_n^*\mu_0$ satisfy $\|\varphi_n^*\mu_0\|_\infty \leq |\lambda_n| \leq \|\lambda\|_\infty$.

Define now for $n\in\N$ the sets
\[ \Delta_n := \zeta_{n-1}\circ\psi^{-1}\left(\{w\colon |w|\leq|\alpha|R\}\right)\text{ and }A_n := \zeta_{n-1}\circ\psi^{-1}\left(\{w\colon |\alpha|R < |w| < R\}\right) \]
and let $g_\lambda\colon\Cx\to\Cx$ be the quasiregular map given by
\[ g_\lambda(z) := \begin{cases}
                    \zeta_n\circ\psi^{-1}\left(\tilde\alpha_n\cdot\psi\circ\zeta_{n-1}^{-1}(z)\right), & z\in\Delta_n, \\
                    \zeta_n\circ\psi^{-1}\left(\varphi_n\circ\psi\circ\zeta_{n-1}^{-1}(z)\right), & z\in A_n, \\
                    f(z) & \text{elsewhere.}
\end{cases} \]
It is important to notice that, because the internal dynamics of $U$ are essentially the dynamics of $h\colon V\to V$, the domains $A_n$ are ``fundamental domains'' for the grand orbit relation (see \cite[pp. 194--195]{MSS83}) of $f$. This means that the orbit of any point in $\Cx$ intersects $A_n$ for at most one value of $n\in\N$, and so the Beltrami coefficient $\mu_\lambda$ given by
\[ \mu_\lambda(z) := \begin{cases}
                        g_\lambda^*\mu_0(z), & z\in A_n, \\
                        (f^j)^*g_\lambda^*\mu_0(z), & z\in f^{-j}(A_n), \\
                        \mu_0(z), & \text{elsewhere}
\end{cases} \]
is $g_\lambda$-invariant, i.e. $g_\lambda^*\mu_\lambda = \mu_\lambda$, and satisfies $\|\mu_\lambda\|_\infty \leq \|\lambda\|_\infty$. Furthermore -- and this is why our construction was so strict -- the map $M\ni\lambda\mapsto \mu_\lambda\in\D$ is analytic for every $z\in\Cx$. It follows from Theorem \ref{lem:MRMT} that there exists a quasiconformal map $\varphi_\lambda\colon\Cx\to\Cx$ fixing $0$ and $z_0$ varying analytically with $\lambda\in M$ and such that
\[ f_\lambda := \varphi_\lambda\circ g_\lambda\circ\varphi_\lambda^{-1} \]
is an entire function.

Now, the fact that both $\varphi_\lambda$ and $g_\lambda$ depend analytically on $\lambda$ does not imply that the same is true of $f_\lambda$. Define the function
\[ \psi_\lambda(z) := \begin{cases}
                        g_\lambda\circ f^{-1}(z), & z\in\Delta_n\cup A_n, \\
                        z, & \text{elsewhere;}
\end{cases} \]
it is immediate that $\psi_\lambda$ is analytic in $\lambda$ and that $g_\lambda = \psi_\lambda\circ f$, so that
\[ f_\lambda = \varphi_\lambda\circ\psi_\lambda\circ f\circ\varphi_\lambda^{-1}. \]
Hence, $(f_\lambda)_{\lambda\in M}$ defines a natural family with basepoint $f_0 = f$ and therefore moves analytically with $\lambda$ (see Theorem \ref{thm:unf}).

Finally, it is clear that $\varphi_\lambda$ conjugates $f|_{J(f)}$ to $f_\lambda|_{J(f_\lambda)}$ by construction, and that $H(\lambda, z) := \varphi_\lambda(z)$ is a holomorphic motion over $M$, implying that $J(f_\lambda)$ moves holomorphically in $(f_\lambda)_{\lambda\in M}$. We must now show that the distortion sequence of $f_\lambda$ at $z_\lambda := \varphi_\lambda(z_0)$ is not constant with $\lambda$, concluding the proof of Theorem \ref{thm:Hermanlong}.

To this end, we start by obtaining Riemann maps $\zeta_{\lambda,n}\colon\D\to U_{\lambda,n}$ with adequate properties. We proceed as in the proof of Theorem \ref{thm:DS}, obtaining integrating maps $h_n\colon\D\to\D$ of $\mu_{\lambda,n} := (\varphi_\lambda\circ\zeta_n)^*\mu_0$ (notice that $h_n$ depends on $\lambda_m$, $m\geq n + 1$, although that dependence is not made explicit). By construction, the functions
\[ \zeta_{\lambda,n}(z) = \varphi_\lambda\circ\zeta_n\circ (h_n)^{-1}(z) \]
are Riemann maps of $U_{\lambda,n}$, and we can define, for $n\in\N$,
\[ g_n(\lambda, z) := \zeta_{\lambda,n}^{-1}\circ f_\lambda\circ\zeta_{\lambda,n-1}(z). \]
It is clear that the functions $g_n$ are holomorphic, move real-analytically with $\lambda\in M$ and are conjugated to $f_\lambda|_{U_{\lambda,n}}$ by the Riemann maps $\zeta_{\lambda,n}$. In particular, $(g_n'(\lambda, 0))_{n\in\N}$ is a distortion sequence for $f_\lambda$ at $z_\lambda$.

It also follows that the integrating maps $h_n$ conjugate $g_n$ to the quasiregular maps
\[ \tilde g_n(\lambda, z) = \begin{cases}
                                \psi^{-1}\left(\tilde\alpha_n\cdot\psi(z)\right), & z\in \psi^{-1}\left(\{w\colon |w|\leq |\alpha|R\}\right), \\
                                \psi^{-1}\left(\varphi_n\circ\psi(z)\right), & z\in \psi^{-1}\left(\{w\colon |\alpha|R < |w|\leq R\}\right).
\end{cases}\]
In particular, the Beltrami coefficients $\mu_{\lambda,n}$ are zero in the disc $\{z\colon |z| \leq |\alpha|R/4\}$, and by the chain rule we have
\[ g_n'(\lambda, 0) = \frac{h_n'(0)}{h_{n-1}'(0)}\tilde\alpha_n, \]
where the derivatives are taken with respect to $z$. Differentiating with respect to $\lambda_n$, and recalling that $h_{n-1}$ is the identity at $\lambda = 0$ and that $h_n$ does not depend on $\lambda_n$, we arrive at
\begin{equation}\label{eq:chain}
\frac{d}{d\lambda_n}g_n'(0, 0) = \rho - |\alpha|\frac{d}{d\lambda_n}h_{n-1}'(0).
\end{equation}
It is clear that we must estimate the derivative on the right-hand side, which will require more insight into the construction of $h_{n-1}$. We recall some elements of Ahlfors and Bers' original work on the parameter dependence of integrating maps (see \cite{AB60}).

To this end, let $\mu\in M(\D)$ be a Beltrami coefficient with integrating map $h:\D\to\D$ fixing $0$ ``and $1$'', and assume that there exists $r \in (0, 1)$ such that $\mu$ is zero in the disc $\{z\colon |z| < r\}$. Consider the Beltrami coefficient
\[ \hat\mu(z) := \begin{cases}
                    \mu(z), & z\in\D, \\
                    \tau^*\mu(z), & z\in\Cx\setminus\overline\D;
\end{cases} \]
recall that $\tau(z) = 1/\bar z$. This Beltrami coefficient is symmetric around the unit circle, and conformal in neighbourhoods of both zero and infinity. It follows from \cite[Theorem V.1]{Ahl06} that there exists an integrating map $F^{\hat\mu}$ of $\hat\mu$ that fixes the origin and infinity and is asymptotic to the identity, i.e.,
\[ F^{\hat\mu}(z) = a_0 + z + O(|z|^{-1}) \text{ for $z\in\Cx$ in a neighbourhood of infinity.} \]
The map $\tilde F^{\hat\mu}(z) = F^{\hat\mu}(z)/F^{\hat\mu}(1)$ is another integrating map of $\hat\mu$, and fixes $0$, $1$, and infinity. Since the Beltrami coefficient $\hat\mu$ is symmetric, it follows from uniqueness of the integrating map that $\tilde F^{\hat\mu}$ is symmetric around the unit circle (i.e., satisfies $\overline{\tilde F}^{\hat\mu}(z) = 1/\tilde F^{\hat\mu}(1/\bar z)$), and hence its restriction to the unit disc is an integrating map of $\mu$. By uniqueness, we have $h \equiv \tilde F^{\hat\mu}|_\D$. Furthermore, by symmetry of $\tilde F^{\hat\mu}$,
\[ h'(0) = \lim_{z\to 0} \frac{\tilde F^{\hat\mu}(z)}{z} = \lim_{z\to 0}\frac{z\overline{F^{\hat\mu}(1)}}{z(1 + z\overline{a_0} + z\overline{O(|z|)}} = \overline{F^{\hat\mu}(1)}. \]
Thus, in order to understand how $h'(0)$ moves with $\mu$, we must understand how $F^{\hat\mu}(1)$ moves with $\hat\mu$.

Given $\vartheta\in L^p(\Cx)$, $p > 2$, let us introduce the integral operators
\[ P\vartheta(s) = -\frac{1}{\pi}\iint \vartheta(z)\left(\frac{1}{z - s} - \frac{1}{z}\right)\,|dz|^2 \]
and
\[ T\vartheta(s) = -\frac{1}{\pi}P.V.\iint \frac{\vartheta(z)}{(z - s)^2}\,|dz|^2, \]
where integrals are taken over the whole plane. It was shown by Ahlfors and Bers (see \cite[Chapter V]{Ahl06}) that
\begin{equation}\label{eq:intop}
     F^{\hat\mu}(z) = z + P[\hat\mu\cdot (\vartheta + 1)](z),
\end{equation}
where
\[ \vartheta = T(\hat\mu\cdot\vartheta) + T\hat\mu = T\hat\mu + (T\hat\mu)^{\circ 2} + \cdots, \]
and $p > 2$ is chosen so that $kC_p < 1$, where $\|\hat\mu\|_\infty\leq k < 1$ and $C_p$ is the $L^p(\Cx)$ norm of $T$ (the fact that $T$ extends as a linear operator to $L^p(\Cx)$ and that $C_p$ is finite was shown earlier by Calder\'on and Zygmund). To tie it all together, we have:
\begin{lemma}\label{lem:frechet}
Let $K\subset\Cx$ be a compact set, and let
\[ \mathcal{B}_K := \{\mu\in L^\infty(\Cx)\colon \|\mu\|_\infty < 1, \mathrm{supp}(\mu)\subset K\}. \]
Consider the map $\mathcal{B}_K\ni\mu\mapsto F^\mu(1)\in\Cx$. Then, its (complex) Fr\'echet differential at the origin is given by
\[ D[F^0(1)](\nu) = P\nu(1). \]
\end{lemma}
\begin{proof}
Since $F^0(z) = z$, we have by (\ref{eq:intop})
\[ \lim_{\|\nu\|_\infty\to 0} \frac{|F^\nu(1) - F^0(1) - P\nu(1)|}{\|\nu\|_\infty} = \lim_{\|\nu\|_\infty\to 0} \frac{|P[\nu(\vartheta_\nu + 1)](1) - P\nu(1)|}{\|\nu\|_\infty}, \]
where $\vartheta_\nu$ satisfies $\vartheta_\nu = T(\nu\vartheta_\nu) + T\nu$. By definition, the proof is complete if we can show that the limit above is zero; we can assume, henceforth, that $\|\nu\|_\infty < 1/4$. First, by linearity of $P$, we have $P[\nu(\vartheta_\nu + 1)](1) = P(\nu\vartheta_\nu + \nu)(1) = P(\nu\vartheta_\nu)(1) + P\nu(1)$, and so
\[ \lim_{\|\nu\|_\infty\to 0} \frac{|F^\nu(1) - F^0(1) - P\nu(1)|}{\|\nu\|_\infty} = \lim_{\|\nu\|_\infty\to 0} \frac{|P(\nu\vartheta_\nu)(1)|}{\|\nu\|_\infty}. \]
Now, by H\"older's inequality, we obtain
\[ |P(\nu\vartheta_\nu)(1)| \leq \frac{1}{\pi}\|\nu\vartheta_\nu\|_p\left\|\frac{1}{z(z - 1)}\right\|_q, \]
where $p > 2$ is fixed and such that $C_p < 2$ and $q = 1/(1 - 1/p)$. Thus,
\[ \lim_{\|\nu\|_\infty\to 0} \frac{|F^\nu(1) - F^0(1) - P\nu(1)|}{\|\nu\|_\infty}\leq \frac{1}{\pi}\left\|\frac{1}{z(z - 1)}\right\|_q\lim_{\|\nu\|_\infty\to 0}\frac{\|\nu\vartheta_\nu\|_p}{\|\nu\|_\infty}. \]
Next, we note that $\|\nu\vartheta_\nu\|_p \leq \|\nu\|_\infty\|\vartheta_\nu\|_p$. It also follows from the definition of $\vartheta_\nu$ (see, for instance, \cite[p. 55]{Ahl06}) that
\[ \|\vartheta_\nu\|_p\leq \frac{C_p}{1 - C_p/4}\|\nu\|_p; \]
we are left with
\[ \lim_{\|\nu\|_\infty\to 0} \frac{|F^\nu(1) - F^0(1) - P\nu(1)|}{\|\nu\|_\infty}\leq \frac{1}{\pi}\left\|\frac{1}{z(z - 1)}\right\|_q\frac{C_p}{1 - C_p/4}\lim_{\|\nu\|_\infty\to 0}\|\nu\|_p. \]
Since $\mathrm{supp}(\nu)\subset K$ by definition, we have $\|\nu\|_p \leq |K|^{1/p}\|\nu\|_\infty$, where $|K|$ denotes the Lebesgue area of $K$. We are done.
\end{proof}

As the culmination of the preceding discussion, Lemma \ref{lem:frechet} allows us to apply the chain rule to (\ref{eq:chain}) and write the following equation:
\begin{equation}\label{eq:newchain}
\frac{d}{d\lambda_n}g_n'(0, 0) = \rho - |\alpha|\overline{P\left(\frac{d\hat\mu_{\lambda,n-1}}{d\lambda_n}\right)(1)},
\end{equation}
where $\hat\mu_{\lambda,n} = \mu_{\lambda,n} + \tau^*\mu_{\lambda,n}$. Calculating the right-hand side is a Herculean task, which we will carry out in Appendix \ref{ap:claims}. We summarise the main steps in the following two claims.
\begin{claim}\label{cl:approx}
Let $\varphi_n(z) := \varphi(\lambda_n, z)$ be the interpolating maps given by Lemma \ref{lem:interpol}, and let $\nu_{n-1} := \varphi_n^*\mu_0$ denote their Beltrami coefficients. Then, if $R > 0$ is chosen sufficiently small,
\[ P\left(\frac{d\hat\mu_{\lambda,n-1}}{d\lambda_n}\right)(1) \approx P\left(\frac{\partial \nu_{n-1}}{\partial\lambda_n} + \frac{\partial(\tau^*\nu_{n-1})}{\partial\overline{\lambda_n}}\right)(1), \]
where $\partial/\partial\lambda_n$ and $\partial/\partial\overline{\lambda_n}$ denote Wirtinger derivatives.
\end{claim}
\begin{claim}\label{cl:calc}
With the notation above,
\[ P\left(\frac{\partial \nu_{n-1}}{\partial\lambda_n} + \frac{\partial(\tau^*\nu_{n-1})}{\partial\overline{\lambda_n}}\right)(1) = \frac{\rho}{|\alpha|}(-1 + 2i). \]
\end{claim}
Combining Claims \ref{cl:approx} and \ref{cl:calc} with (\ref{eq:newchain}), we arrive at
\[ \frac{d}{d\lambda_n}g_n'(0, 0)\approx 2\rho(1 + i) \neq 0, \]
completing the proof of Theorem \ref{thm:Hermanlong}.

\appendix
\section{Proving Claims \ref{cl:approx} and \ref{cl:calc}}\label{ap:claims}
Here, we carry out the remaining calculations for the proof of Theorem \ref{thm:Hermanlong}. For ease of notation, we assume now without loss of generality that $\alpha > 0$ (if not, we can rotate the Riemann maps $\zeta_n$ to make it so).
\begin{proof}[Proof of Claim \ref{cl:approx}]
We recall the linearising coordinates $\psi\colon\Delta\to\D_L$ of $g$ (see page 21). From the construction, the dependence of the Beltrami coefficients $\hat\mu_{\lambda,n-1} = \mu_{\lambda,n-1} + \tau^*\mu_{\lambda,n-1}$ on $\lambda_n$ has two factors: one inside the unit disc given by $\mu_{\lambda,n-1} = \psi^*\nu_{n-1}$, and another outside the unit disc given by $\tau^*\mu_{\lambda,n-1} = \tau^*\psi^*\nu_{n-1}$. The inside factor depends only on $\lambda_n$, while the outside factor depends only on $\overline{\lambda_n}$. In other words, we have the factorisation
\begin{equation}\label{eq:factor}
    \frac{d}{d\lambda_n}\hat\mu_{\lambda,n-1}(z) = \begin{cases} \frac{\overline{\psi'(z)}}{\psi'(z)}\frac{\partial}{\partial\lambda_n}\nu_{n-1}(\psi(z)), & z\in\psi^{-1}(\{w : \alpha R < |w| < R\}), \\
 \frac{\psi'(z)z^2}{\overline{\psi'(z)}\bar z^2}\overline{\frac{\partial}{\partial\lambda_n}\nu_{n-1}\left(\psi\left(\frac{1}{\bar z}\right)\right)}, & 1/\bar z\in\psi^{-1}(\{w : \alpha R < |w| < R\}), \\
 0 & \text{elsewhere.} \end{cases}
\end{equation}
 Let us focus on the case $z\in\psi^{-1}\left(\{w\colon \alpha R < |w| < R\}\right)$; the other case is similar, and the calculations require little modification. If we write $\arg \psi'(z) = \theta(z)$, the relevant case of (\ref{eq:factor}) becomes
 \[ \frac{d}{d\lambda_n}\hat\mu_{\lambda,n-1}(z) = \frac{d}{d\lambda_n}\mu_{\lambda,n-1}(z) = \exp\left(-2\theta(z)\right)\frac{\partial}{\partial\lambda_n}\nu_{n-1}\left(\psi(z)\right); \]
 now, let $\epsilon(z) = \exp(-2\theta(z)) - 1$. Then, for $z\in\psi^{-1}\left(\{w\colon \alpha R < |w| < R\}\right)$,
 \[ \frac{d}{d\lambda_n}\hat\mu_{\lambda,n-1}(z) = \frac{d}{d\lambda_n}\mu_{\lambda,n-1}(z) = \left(1 + \epsilon(z)\right)\frac{\partial}{\partial\lambda_n}\nu_{n-1}\left(\psi(z)\right), \]
 and we can explicitly calculate $\nu_{n-1}$ and its derivative to obtain
 \[ \frac{d}{d\lambda_n}\mu_{\lambda,n-1}(z) = \left(1 + \epsilon(z)\right)\frac{\rho te^{2i\theta}}{2\alpha R(1 - \alpha)}, \]
 where $\psi(z) = te^{i\theta}$. Since $\alpha R < t < R$ in the case we are considering, we can invoke the linearity of $P$ and H\"older's inequality to show that
 \[ \left|P\left(\frac{d\mu_{\lambda,n-1}}{d\lambda_n}\right)(1) - P\left(\frac{\partial\nu_{n-1}}{\lambda_n}\right)(1)\right| \leq \frac{1}{\pi}\frac{\rho}{2\alpha(1 - \alpha)}\|\epsilon\|_p\left\|\frac{1}{z(1 - z)}\right\|_q; \]
 the continuity of $\psi'$ now implies that if $R > 0$ is chosen small enough then $|\epsilon(z)|$ can be made arbitrarily small\footnote{We could obtain more sophisticated estimates for $|\epsilon(z)|$ by appealing to Koebe's distortion theorem and the Borel-Carath\'eodory inequality.} since $\psi'(0) = 1$, finishing the proof for the case $z\in\psi^{-1}\left(\{w\colon \alpha R < |w| < R\}\right)$. The case $1/\bar z\in\psi^{-1}\left(\{w\colon \alpha R < |w| < R\}\right)$ follows a similar reasoning; we omit the details.
\end{proof}
We remark that, from the proof of Claim \ref{cl:approx}, how small $R$ must be chosen depends only on $\alpha$. All that is left is to calculate some integrals.
\begin{proof}[Proof of Claim \ref{cl:calc}]
By linearity of $P$,
\[ P\left(\frac{\partial \nu_{n-1}}{\partial\lambda_n} + \frac{\partial(\tau^*\nu_{n-1})}{\partial\overline{\lambda_n}}\right)(1) = P\left(\frac{\partial \nu_{n-1}}{\partial\lambda_n}\right)(1) + P\left(\frac{\partial(\tau^*\nu_{n-1})}{\partial\overline{\lambda_n}}\right)(1). \]
We have already seen in the proof of Claim \ref{cl:approx} that
\[ \left.\frac{\partial}{\partial\lambda_n}\nu_{n-1}(te^{i\theta})\right|_{\lambda_n=0} = \frac{\rho te^{i\theta}}{2\alpha R(1 - \alpha)}; \]
we henceforth omit the specification that the derivative is being evaluated at $\lambda_n = 0$ for ease of notation. By changing to polar coordinates, the first integral in question now becomes
\[ P\left(\frac{\partial \nu_{n-1}}{\partial\lambda_n}\right)(1) = -\frac{\rho}{2\pi\alpha R(1 - \alpha)}\int_0^{2\pi}\int_{\alpha R}^R \left(1 + \frac{1}{te^{i\theta} - 1}\right)\,dtd\theta. \]
Linearity of the integral now yields
\[ P\left(\frac{\partial \nu_{n-1}}{\partial\lambda_n}\right)(1) = -\frac{\rho}{2\pi\alpha R(1 - \alpha)}\left(2\pi R(1 - \alpha) + \int_0^{2\pi} e^{-i\theta}\log\frac{Re^{i\theta} - 1}{\alpha Re^{i\theta} - 1}\,d\theta\right), \]
where $\log$ denotes the principal branch of the logarithm -- notice that, since $R$ is small and $0 < \alpha < 1$, both $\theta\mapsto Re^{i\theta} - 1$ and $\theta\mapsto \alpha Re^{i\theta} - 1$ are non-zero, and in fact describe circles that do not surround the origin. Defining the function
\[ r(z) = \frac{1}{z^2}\log\frac{Rz - 1}{\alpha R - 1}, \]
it follows from the residue theorem that
\[ \int_0^{2\pi} e^{-i\theta}\log\frac{Re^{i\theta} - 1}{\alpha Re^{i\theta} - 1}\,d\theta = \int_{|z| = 1} r(z)\,dz = -2\pi iR(1 - \alpha), \]
and thus
\[ P\left(\frac{\partial \nu_{n-1}}{\partial\lambda_n}\right)(1) = -\frac{\rho}{2\pi\alpha R(1 - \alpha)}\left(2\pi R(1 - \alpha) - 2\pi iR(1 - \alpha)\right) = \frac{\rho}{\alpha}(-1 + i). \]

The remaining integral is
\[ P\left(\frac{\partial(\tau^*\nu_{n-1})}{\partial\overline{\lambda_n}}\right)(1), \]
which we calculate to be
\[ P\left(\frac{\partial(\tau^*\nu_{n-1})}{\partial\overline{\lambda_n}}\right)(1) = -\frac{\rho}{2\alpha R(1 - \alpha)}\int_0^{2\pi}e^{i\theta}\int_{1/R}^{1/(\alpha R)}\left(\frac{e^{i\theta}}{te^{i\theta} - 1} - \frac{1}{t}\right)\,dtd\theta. \]
Evaluating the integral in $t$ gives
\[ P\left(\frac{\partial(\tau^*\nu_{n-1})}{\partial\overline{\lambda_n}}\right)(1) = -\frac{\rho}{2\alpha R(1 - \alpha)}\int_0^{2\pi} e^{i\theta}\left(\log\frac{(\alpha R)^{-1}e^{i\theta} - 1}{R^{-1}e^{i\theta} - 1} + \log\alpha\right)\,d\theta, \]
which (by Cauchy's theorem) simplifies to
\[ P\left(\frac{\partial(\tau^*\nu_{n-1})}{\partial\overline{\lambda_n}}\right)(1) = -\frac{\rho}{2\alpha R(1 - \alpha)}\int_0^{2\pi} e^{i\theta}\log\frac{(\alpha R)^{-1}e^{i\theta} - 1}{R^{-1}e^{i\theta} - 1}\,d\theta. \]
We introduce the function
\[ r(z) = \frac{1}{z^2}\log\frac{(\alpha Rz)^{-1} - 1}{(Rz)^{-1} - 1}, \]
and apply the residue theorem to a curve going clockwise around the unit circle:
\[ \int_0^{2\pi} e^{i\theta}\log\frac{(\alpha R)^{-1}e^{i\theta} - 1}{R^{-1}e^{i\theta} - 1}\,d\theta = -\int_{|z| = 1}r(z)\,dz = -2\pi iR(1 - \alpha), \]
and thus
\[ P\left(\frac{\partial(\tau^*\nu_{n-1})}{\partial\overline{\lambda_n}}\right)(1) = \frac{\rho}{\alpha}i. \]
This completes the proof.
\end{proof}

\end{document}